\documentclass[11pt]{amsart}

\usepackage[margin=1.12in]{geometry}
\usepackage{amsmath,amssymb,mathtools}
\usepackage{enumitem}
\usepackage{xcolor}
\usepackage{microtype}
\usepackage[colorlinks=true,linkcolor=blue!55!black,
  citecolor=blue!55!black,urlcolor=blue!55!black]{hyperref}

\numberwithin{equation}{section}

\newtheorem{theorem}{Theorem}[section]
\newtheorem{proposition}[theorem]{Proposition}
\newtheorem{lemma}[theorem]{Lemma}
\newtheorem{corollary}[theorem]{Corollary}
\theoremstyle{definition}
\newtheorem{definition}[theorem]{Definition}
\newtheorem{remark}[theorem]{Remark}

\newcommand{\dd}{\,\mathrm d}
\newcommand{\ddcJ}{\mathrm d\mathrm d_J^c}
\newcommand{\ii}{\sqrt{-1}}

\newcommand{\cM}{\mathcal M}

\newcommand{\cD}{\mathcal D}

\newcommand{\T}{\mathbb T}
\newcommand{\TC}{\T^{\mathbb C}}

\newcommand{\KRS}{\mathrm{KRS}}
\newcommand{\KE}{\mathrm{KE}}
\newcommand{\tr}{\operatorname{tr}}
\newcommand{\Ric}{\operatorname{Ric}}
\newcommand{\Vol}{\operatorname{Vol}}

\newcommand{\Aut}{\operatorname{Aut}}

\makeatletter
\let\gkrf@setauthors\@setauthors
\renewcommand{\@setauthors}{%
  \begingroup
  \let\MakeUppercase\@firstofone
  \gkrf@setauthors
  \endgroup
}
\makeatother

\title[GKRF on toric Fano manifolds]
{Convergence of Generalized K\"ahler--Ricci Flow
on Toric Fano Manifolds}

\author[L. Huang and X. Zhu]{Liding Huang$^{*}$ and Xiaohua Zhu$^{\ddag}$}
		\address{School of Mathematical Sciences, Xiamen University, Xiamen 361005, P. R. China}
	\email{huangliding@xmu.edu.cn}
	\address{School of Mathematical Sciences,  Capital Normal University,  Beijing,  100048, P. R. China }
	\email{xhzhu@math.pku.edu.cn}

	 \thanks {$*$ partially supported  by NSFC 12401075.}
\thanks {$\ddag$ partially supported  by National Key R\&D Program of China   2023YFA1009900  and  NSFC 12271009.}


\begin{document}

\begin{abstract}We  first define an extended  Mabuchi  functional  for a class of  generalized  K\"ahler metrics of symplectic type 
 on a  toric Fano manifold $(M,J,\T)$, and prove it is proper.
Then by showing the 
 monotonicity of functional  along the
normalized generalized K\"ahler--Ricci flow for any  initial $\T$-invariant
 generalized  K\"ahler metric with its symplectic form  in  $2\pi c_1(M,J)$, we prove that  the flow  converges to a   K\"ahler--Ricci soliton.   The result  confirms   a  conjecture of Apostolov--Streets--Ustinovskiy.

\end{abstract}

\maketitle

\section{Introduction}

Generalized K\"ahler geometry originated in the work of
Gates--Hull--Ro\v{c}ek on two-dimensional supersymmetric nonlinear
sigma models~\cite{GatesHullRocek}.
Hitchin later introduced generalized complex geometry  \cite{HitchinGC}.
Gualtieri developed this theory and identified
generalized K\"ahler geometry with bi-Hermitian geometry \cite{GualtieriGC,GualtieriGK}.
 

In bi-Hermitian notation, a generalized K\"ahler structure consists
of a Riemannian metric $g$, two
integrable $g$-Hermitian complex structures $I$ and $J$, and a closed
torsion three-form $\mathcal H$, which  satisfy
\[
 d_I^c\omega_I=\mathcal H=-d_J^c\omega_J,
 \qquad d\mathcal H=0.
\]
For a structure of symplectic type, $I+J$ is invertible and, in the
splitting used here, $\mathcal H=db$ for a globally defined two-form
$b$.  The structure is equivalently encoded by a symplectic form $F$
taming $J$:
\[
 I=-F^{-1}J^*F,\qquad -FJ=g+b.
\]

Streets--Tian introduced pluriclosed flow as a non-K\"ahler analogue of
K\"ahler-Ricci (abbreviated as KR)  flow and showed that, after suitable gauge
transformations, it preserves generalized K\"ahler structures~\cite{StreetsTianPCF,StreetsTian}.  This flow can be used to
study canonical metrics on non-K\"ahler manifolds and moduli problems
in generalized K\"ahler geometry
\cite{GualtieriGK,ASU-space}. 
In the fixed-$J$ gauge,   the normalized  flow  was reformulated  by  Apostolov--Streets--Ustinovskiy 
 as a generalized KR flow~\cite{ASU-toric}:
\begin{equation}\label{eq:intro-normalized-GKRF}
 \partial_t\omega_J(t)
 =-2(\rho^B_J)^{1,1}+2\lambda\omega_J,
 \qquad
 \partial_t I
 =\mathcal L_{\theta_I^{\sharp}
                 -\theta_J^{\sharp}}I.
\end{equation}
Here $\rho^B_J$ is the Bismut--Ricci form, $\theta_I,\theta_J$ are
the Lee forms, and $\sharp$ is taken with respect to $g$, in the
conventions of Apostolov--Streets--Ustinovskiy~\cite{ASU-toric}.

 In case that $(M,J)$ is a Fano manifold,  we use the Fano normalization $\lambda=1$ in
\eqref{eq:intro-normalized-GKRF}.  Then  the symplectic form class is preserved    along the flow  if   the   symplectic form $F$ of  initial generalized K\"ahler metric  is chosen in 
$2\pi c_1(M,J)$. In particular,  it reduces to the normalized
KR  flow 
 when $I=J$ and $b=0$.


On a toric Fano manifold $(M^m, J, \T)$, where $\T$ is the $m$-dimensional maximal compact torus, 
we  may consider the normalized generalized KR flow  (\ref{eq:intro-normalized-GKRF}) for   an initial $\T$-invariant
 generalized  K\"ahler metric with its symplectic form  in  $2\pi c_1(M,J)$.
{We write $\mathfrak t=\operatorname{Lie}(\T)$,
$\mathfrak t^*$ for its dual, and $\TC$ for the complexification of $\T$.}
Then \eqref{eq:intro-normalized-GKRF} can be reduced to   the following  parabolic scalar
equation for a family of symplectic potentials $u_t$ on the
corresponding Delzant polytope $P$ of $(M,J)$~\cite{ASU-toric}:
\begin{equation}\label{eq:intro-flow}
 \dot u
 =\log\det(D^2u+\ii B_t)+2u-2x\cdot\nabla u,
 \qquad B_t=e^{-2t}B_0,
\end{equation}
where $B_0\in\wedge^2\mathfrak t$ is constant.  The type-$A$
deformation parameter also satisfies $A_t=e^{-2t}A_0$ and is
decoupled from \eqref{eq:intro-flow}.   Apostolov--Streets--Ustinovskiy have proved  that  there exists  a solution $u=u_t$ of (\ref{eq:intro-flow})  for any $t>0$  \cite{ASU-toric}.  

The data $(u_t,A_t,B_t)$ determine a family of generalized
K\"ahler structures $(g_t,I_t,J_t,b_t)$.  We write $J$ for
the initial complex structure $J_0$.  Let
$\Phi_t:(M,J)\to(M,J_t)$ be the canonical equivariant
biholomorphisms satisfying $\Phi_t^*J_t=J$.  The pulled-back family
\[
 (\widetilde g_t,\widetilde I_t,J,\widetilde b_t)
 :=\Phi_t^*(g_t,I_t,J_t,b_t)
\]
solves \eqref{eq:intro-normalized-GKRF} on $(M,J)$
\cite[Proposition~4.1]{ASU-toric}.
Set $\omega_t:=g_t(J_t\cdot,\cdot)$,
then the Hermitian form on $(M,J)$ 
\[
\tilde{\omega}_t:=\widetilde g_t(J\cdot,\cdot)
 =\Phi_t^*\omega_t,
\]
 is the solution of \eqref{eq:intro-normalized-GKRF}.
Based on the existence of KR solitons on toric
Fano manifolds~\cite{WangZhu}, Apostolov--Streets--Ustinovskiy
also conjectured that $\tilde{\omega}_t$ converges to a KR soliton
after pullbacks by holomorphic automorphisms
\cite[Conjecture~2.12]{ASU-toric}.  We note that in the K\"ahler case,  the convergence of 
normalized KR flow 
from a $\T$-invariant K\"ahler metric 
 on a toric
Fano manifold has been proved 
  by the second author before \cite{Zhu}.

In this paper, we  shall  prove the above conjecture.   Precisely, we have

\begin{theorem}\label{thm:main}
 Let $(M,J,\T)$ be a toric Fano manifold.  Suppose that
 the initial datum is a $\T$-invariant generalized K\"ahler metric 
 of symplectic type whose symplectic form lies in
 $2\pi c_1(M,J)$.  Then  there 
 exist
 a smooth family $\gamma_t\in\TC$ and a KR soliton
 $\omega_{\KRS}$ such that
 \begin{equation}\label{eq:main-KRS-convergence}
  (\gamma_t^*\tilde{\omega}_t,
      \gamma_t^*\widetilde I_t,\gamma_t^*\widetilde b_t)
  \longrightarrow(\omega_{\KRS},J,0)
 \end{equation}
 in $C^\infty$ as $t\to\infty$.  
 Moreover, if the Futaki invariant vanishes, then there
 is a K\"ahler--Einstein metric
 $\omega_{\KE,\infty}$ for which \eqref{eq:main-KRS-convergence}
 holds with $\gamma_t=\operatorname{id}$ and
 $\omega_{\KRS}=\omega_{\KE,\infty}$.  In this case, the convergence
 is exponential in every $C^k$-norm.
\end{theorem}


To prove Theorem~\ref{thm:main},  we  first introduce an extended  Mabuchi  functional  $\cM_\xi(u,B)$
 for  generalized  K\"ahler metrics of   model $(u,0,B)$ on the  toric Fano manifold $(M,J,\T)$, see Section~\ref{sec:mabuchi}.  Here $\xi\in\mathfrak t\cong\mathbb R^m$ is associated to the KR soliton vector field on $(M,J)$.  This functional  is  just the modified $K$-energy $\cM_\xi(u,0)$
  when $B=0$ on the toric Fano manifold, as in
  \cite{TianZhu,WangZhouZhu}.  Actually, it can be compared with  $\cM_\xi(u,0)$
  such that the properness  of  $\cM_\xi(u,B)$
 still holds as $\cM_\xi(u,0)$, see Proposition~\ref{prop:M-variational-structure}.
It is interesting to mention that  the  properness and coercivity results of  modified $K$-energy  modulo
automorphisms hold on  Fano manifolds  which admit KR
solitons, see Darvas--Rubinstein~\cite{DarvasRubinstein} and
Hisamoto~\cite[Section~4]{HisamotoToric}.
 Moreover,  we can show that
 $\cM_\xi(u_t, B_t)$ is monotone along (\ref{eq:intro-flow})  as in the case of KR flow (cf. \cite{TZ1, TZZZ}), see Theorem \ref{thm:dissipation}.
As a consequence,  there exists a sequence of
  $u_j$, which  converges to $u_\infty$ in $L^1(P)$.  The latter implies that 
the relative  K\"ahler potential $\varphi_\infty$ associated to  $u_\infty$ belongs to the $\mathcal E^1$-space,
  see Proposition~\ref{prop:toric-compactness}.

Next, we shall improve that  $u_j$ is in fact  convergent in  $C^k$-norm for any integer $k>0$.
Then by a stability result in \cite{HuangZhu}, we will finish the proof
of Theorem~\ref{thm:main}.
Inspired  by  the work  of  Di Nezza--Lu  on the
uniqueness of  weak KR flow
with an  initial  K\"ahler potential in $\mathcal E^1$-space  \cite{DNL}, we  actually establish the following a priori estimates  for the reduced equation (\ref{eq:intro-flow}).


\begin{theorem}
\label{thm:conditional-weak-B-flow}
Fix $B_0\in\wedge^2\mathfrak t$.  Let $u_{0,j}\in\mathcal S(P)$ be
a sequence of symplectic potentials corresponding to toric
K\"ahler metrics in the class $2\pi c_1(M,J)$,
normalized by $u_{0,j}\geq0$ and $u_{0,j}(0)=0$, and converging to
$u_\infty$ in $L^1(P)$.  Let $\varphi_{0,j}$ and $\varphi_0$ be
the corresponding relative K\"ahler potentials:
\[
 \varphi_{0,j}=u_{0,j}^*-\bar\phi_0,\qquad
 \varphi_0=u_\infty^*-\bar\phi_0\in\mathcal E^1_\T,
\]
where $\bar\phi_0$ is a fixed reference Legendre potential.
Assume that each triple $(u_{0,j},0,B_0)$ determines a smooth toric
generalized K\"ahler structure of symplectic type on $M$, and that
the corresponding bi-Hermitian data satisfy
\begin{equation}\label{eq:weak-B-uniform-gap}
 \sup_j\sup_M|b_{0,j}|_{g_{0,j}}^2\leq2Q
\end{equation}
for some $Q<\infty$.  

Let $u_j(t)$ solve \eqref{eq:intro-flow}
with initial potential $u_{0,j}$, and let $\varphi_j(t)$ be the
corresponding relative K\"ahler potentials  (which satisfy
\eqref{eq:manifold-scalar-flow} in Section 2 below) with initial values
$\varphi_j(0)=\varphi_{0,j}$.
Then the following is true:

i)  For every
$0<\varepsilon<T<\infty$ and every $k\geq0$,
\[
 \|\varphi_j\|_{C^k([\varepsilon,T]\times M)}
 \leq C_{\varepsilon,T,k}.
\]
The constants $C_{\varepsilon,T,k}$ are independent
of $j$. Moreover, $\varphi_j$ converges in
$C^\infty_{\mathrm{loc}}((0,\infty)\times M)$ to a smooth solution
$\varphi(t)$ with $B_t=e^{-2t}B_0$.  

ii) The corresponding symplectic
potential $u(t):=(\bar\phi_0+\varphi(t))^*$ satisfies
\[
 \|u(t)-u_\infty\|_{L^1(P)}
 \longrightarrow0\quad(t\to0^+).
\]

iii) 
The limit is independent of the approximating sequence among all
sequences satisfying a bound of the form
\eqref{eq:weak-B-uniform-gap}.  Thus every $\varphi_0$ admitting such
an approximation determines a unique weak flow within this class of
approximation limits, smooth for $t>0$.
\end{theorem}

\vskip5mm


The paper is organized as follows.  In Section~2 we  recall some basic  results  for the  generalized  KR   flow  on  toric Fano manifolds by Apostolov--Streets--Ustinovskiy  in \cite{ASU-toric}.
In  Section~3  we first introduce an extended  Mabuchi functional and  then we prove the properness and the monotonicity of the functional  along (\ref{eq:intro-flow}), see 
Proposition \ref{prop:M-variational-structure} and  Theorem \ref{thm:dissipation}. The main results, 
 Theorem \ref{thm:conditional-weak-B-flow}  and Theorem \ref{thm:main}  will be proved in Section 4 and Section 5, respectively.

\vskip5mm

$\mathbf {Acknowledgment.}$  The paper was partially  finished  while the second author was visiting ICTP.   He would like to thank  professor  Claudio Arezzo  and ICTP for the invitation and  the hospitality.

\vskip5mm
\section{Preliminaries}\label{sec:preliminaries}

In this section we recall some basic results for the
generalized KR flow on toric Fano manifolds, see \cite{ASU-toric}
for details.


\subsection{Toric structures and the scalar equation}

For a  complex structure $J$, we set
\[
 d_J^c=\ii(\bar\partial_J-\partial_J),\qquad
 \omega_J(V,W)=g(JV,W).
\]
Thus $\ddcJ=2\ii\partial_J\bar\partial_J$, and
\[
 \Ric(\eta)=-\frac12\ddcJ\log\det(\eta_{i\bar j}),\qquad
 \Delta_g^C=2g^{\bar ji}\nabla_i^C\nabla_{\bar j}^C.
\]
We recall the definition of a generalized K\"ahler structure of
symplectic type~\cite[Section~2.2]{ASU-toric}.

\begin{definition}
\label{def:symplectic-GK}
Let $(M,J)$ be a compact complex manifold.  A generalized
K\"ahler structure of symplectic type is a triple $(g,I,J)$, where
$g$ is a Riemannian metric Hermitian with respect to both $I$ and
$J$, $I$ is integrable, and
\[
 d_I^c\omega_I=\mathcal H=-d_J^c\omega_J,
 \qquad d\mathcal H=0,
 \qquad \det(I+J)\ne0\quad\text{on }M.
\]
\end{definition}

The associated symplectic form $F$ and the
two-form $b$ in the canonical splitting satisfy~\cite[Section~2.2]{ASU-toric}:
\[
 F=-2g(J+I)^{-1},\qquad
 b=-g(J+I)^{-1}(I-J),\qquad -FJ=g+b,
 \qquad db=-d_J^c\omega_J.
\]
The form $F$ is closed and tames $J$: $F(V,JV)>0$ for every
nonzero tangent vector $V$.
Products in these identities denote composition of bundle maps.
In particular, $(g+b)(V,W)=-F(JV,W)$.


Since $[F_0]=2\pi c_1(M,J)$, we use the anticanonical normalization
of the labelled Delzant polytope:
\[
 L_r(x)=\langle x,\nu_r\rangle+1,\qquad
 P=\{x\in\mathfrak t^*:L_r(x)\geq0,\ 1\leq r\leq d\},
\]
where $\nu_r$ is the primitive inward normal to the $r$th facet.
In particular, $0\in P^\circ$.  Let $\xi\in\mathfrak t$ be the
soliton vector, normalized by
\begin{equation}\label{eq:barycenter}
 \rho_\xi(x)=e^{\langle\xi,x\rangle},
 \qquad \int_Px\,\rho_\xi(x)\dd x=0.
\end{equation}
This is the vector associated to the KR soliton
uniquely determined by the vanishing of generalized Futaki-invariant
\cite{WangZhu,TianZhu}.

Fix a lattice basis $e_1,\ldots,e_m$ of $\mathfrak t$.  We write
$K_\alpha$ for the fundamental vector field generated by $e_\alpha$
and $x_\alpha$ for the corresponding moment-map component.
Fix the canonical
Guillemin potential
\[
 u_c=\frac12\sum_{r=1}^d L_r\log L_r.
\]
The Guillemin class $\mathcal S(P)$ consists of the functions
$u=u_c+v$, with $v\in C^\infty(\bar P)$, such that $D^2u$ is
positive on $P^\circ$ and remains positive on the tangent space of
the relative interior of every face.  These facewise positivity
conditions are equivalent to smooth compactification of the toric
K\"ahler metric determined by $U=D^2u$
\cite{Guillemin,Abreu}.  Adding an affine function to $u$ does not
change this metric.  Throughout, $\dd x$ denotes the lattice
Lebesgue measure on $\mathfrak t^*$.

A smooth toric generalized K\"ahler structure of
symplectic type is determined, up to equivariant identification, by
data $(u,A,B)\in\mathcal S(P)\times\wedge^2\mathfrak t
\times\wedge^2\mathfrak t$ satisfying the compactification
criterion~\cite[Lemma~6.3]{ASU-space}. In particular,
\begin{equation}\label{eq:positivity}
 U+\ii B>0\quad\text{on }P^\circ,
\end{equation}
and the same positivity holds after restriction to every face.
The parameters $A$ and $B$ are the type-$A$ and type-$B$
deformations, respectively; see
\cite{ASU-toric,Wang-I,Wang-II}.
\begingroup
For the evolving data $(u_t,A_t,B_t)$, set
$U_t=D_x^2u_t$ and $\Psi_t=U_t+A_t$.
Use momentum--angle coordinates $(x,\theta)$ of the underlying
toric K\"ahler structure on the dense orbit. Relative to the ordered
frame $(\partial_{\theta_1},\ldots,\partial_{\theta_m},
\partial_{x_1},\ldots,\partial_{x_m})$, the reconstruction formulas are
\cite[equations~(3.16) and~(3.28)]{ASU-toric}
\[
 J_t\sim
 \begin{pmatrix}
  0&\Psi_t^{\mathsf T}\\
  -(\Psi_t^{\mathsf T})^{-1}&0
 \end{pmatrix}
 \quad\text{and}\quad
 I_t\sim
 \begin{pmatrix}
  2B_t\Psi_t^{-1}&\Psi_t+4B_t\Psi_t^{-1}B_t\\
  -\Psi_t^{-1}&-2\Psi_t^{-1}B_t
 \end{pmatrix},
\]
where $\mathsf T$ denotes transpose.
The Hermitian form $\omega_t$ has matrix
\[
 \omega_t\sim
 \begin{pmatrix}
  \Psi_t^{-1}B_t(\Psi_t^{\mathsf T})^{-1}&-\Psi_t^{-1}U_t\\
  U_t(\Psi_t^{\mathsf T})^{-1}&B_t
 \end{pmatrix}.
\]
\endgroup

The evolution of $A_t$ is decoupled from the
	equations for $u_t$ and $B_t$. We therefore derive the scalar
	estimates in the $A=0$ model~\cite[Corollary~4.3]{ASU-toric}.

The Legendre transform $\phi=u^*$ satisfies
\[
 y=\nabla u(x),\qquad x=\nabla\phi(y),\qquad
 u(x)+\phi(y)=\langle x,y\rangle.
\]
We reserve $\varphi$ for the relative K\"ahler potential on $M$.
Let $\operatorname{Aff}(P)$ act on $\mathcal S(P)$ by addition.  We
write
\[
 \mathcal G_\T(P)
 :=\left\{(u,B)\in\mathcal S(P)\times\wedge^2\mathfrak t:
 (u,0,B)\text{ satisfies the compactification criterion}\right\}
 \big/\operatorname{Aff}(P).
\]
\begingroup
Thus $\mathcal G_\T(P)$ parametrizes the $A=0$ toric generalized
K\"ahler metrics of symplectic type. It is convex, since positivity
of $D^2u+\ii B$, including its restrictions to the faces, is preserved
under convex combinations.

Apostolov--Streets--Ustinovskiy derived the scalar equation
\eqref{eq:intro-flow} for the symplectic potential $u$ from the
generalized KR flow \eqref{eq:intro-normalized-GKRF}.
We also have the following reduced scalar equation for the
Legendre transform $\phi$ of $u$
\cite[Propositions~4.1 and~4.2]{ASU-toric}.
\endgroup

\begin{proposition}\label{prop:ASU-scalar-reduction}
 Let $u_t$ be a smooth solution of \eqref{eq:intro-flow}.
 Its Legendre transform $\phi_t=u_t^*$ satisfies
\[
  \partial_t\phi_t
  =-\log\det\!\left((D^2\phi_t)^{-1}+\ii B_t\right)
   +2\phi_t.
 \]
\end{proposition}

Fix a smooth toric reference form $\omega_0\in2\pi c_1(M,J)$,
and let $\bar\phi_0$ be its K\"ahler potential on the dense orbit.
The relative potential $\varphi_t=\phi_t-\bar\phi_0$
extends smoothly to $M$. We write $\bar\omega_t$ for the
corresponding K\"ahler reduction on $(M,J)$ and set
\[
 \bar\omega_t:=\omega_0+\ddcJ\varphi_t.
\]
\begingroup
In the $A=0$ model, use logarithmic coordinates $y$ and angular
coordinates $\theta$ on the dense orbit of $(M,J)$, and set
$H_t=D_y^2\phi_t$.
Then $\bar\omega_t=\sum_{a,b}(H_t)_{ab}\,dy_a\wedge d\theta_b$, and
\begin{equation}\label{eq:toric-Hermitian-form}
 \tilde{\omega}_t
 =\bar\omega_t+\frac12\sum_{a,b}(H_tB_tH_t)_{ab}
       (dy_a\wedge dy_b+d\theta_a\wedge d\theta_b).
\end{equation}
\endgroup
Let $h_0$ be the Ricci potential determined by
$\ddcJ h_0=\Ric(\omega_0)-\omega_0$, normalized by
$\int_Mh_0\,\omega_0^m=0$.
\begingroup
Using the metric formulas, we find
\[
 q_t:=\log\frac{\bar\omega_t^m}{\tilde{\omega}_t^m}
 =-\log\det(I+\ii H_tB_t)\geq0,
\]
where the inequality follows from \eqref{eq:skew-normal-form}.
The following equivalent formulation of
Proposition~\ref{prop:ASU-scalar-reduction} was proved in
\cite[Proposition~4.4]{ASU-toric}.
\endgroup

\begin{proposition}
\label{prop:manifold-scalar-reduction}
 After changing $\varphi_t$ by a function of time, the K\"ahler
 reduction satisfies
 \begin{equation}\label{eq:manifold-scalar-flow}
  \partial_t\varphi_t
  =\log\frac{\bar\omega_t^m}{\omega_0^m}
   +q_t+2\varphi_t-2h_0.
 \end{equation}
\end{proposition}


\subsection{Matrix identities and local estimates}

Throughout the paper we use
\[
 (U+\ii B)^{-1}=X+\ii Y,\qquad
 \hat C=YB=XU-I.
\]
In matrix identities, $I$ denotes the identity matrix; elsewhere it
denotes the complex structure.  Here $X$ is real symmetric and $Y$
is real skew-symmetric.
Define $S=U^{-\frac{1}{2}}BU^{-\frac{1}{2}}$.
Then $S$ is real skew-symmetric and \eqref{eq:positivity} is
equivalent to $I+\ii S>0$.  Let $s_\alpha\geq0$ denote the parameters
in the orthogonal skew-normal form.  Thus there is $Q\in O(m)$ such
that
\begin{equation}\label{eq:skew-normal-form}
 Q^TSQ
 =\bigoplus_\alpha
   \begin{pmatrix}0&-s_\alpha\\ s_\alpha&0\end{pmatrix}
  \oplus 0.
\end{equation}
The eigenvalues of $I+\ii S$ are $1\pm s_\alpha$ and possibly $1$.
Hence $I+\ii S>0$ if and only if $s_\alpha<1$ for every $\alpha$, and
\[
 \det(I+\ii S)=\prod_\alpha(1-s_\alpha^2)>0.
\]
Direct inversion gives
\begin{equation}\label{eq:X-formula}
 X=U^{-\frac{1}{2}}(I+S^2)^{-1}U^{-\frac{1}{2}}.
\end{equation}
Since $S^2\leq0$ and $I+S^2>0$, formula~\eqref{eq:X-formula}
implies
\begin{equation}\label{eq:D-positive}
 X-U^{-1}\geq0,\qquad (X-U^{-1})U=\hat C.
\end{equation}

\begingroup
We recall the local estimates used to control higher derivatives.
The two-form $b$ has type $(2,0)+(0,2)$ with respect to $J$.
With $\pi_J^{2,0}$ denoting projection onto $(2,0)$-forms, put
\begin{equation}\label{eq:JS-beta-conversion}
 \beta^{\rm JS}=-\ii\pi_J^{2,0}b,\qquad
 \bar\partial\beta^{\rm JS}=\partial\omega_J.
\end{equation}
\endgroup
Apostolov--Streets--Ustinovskiy proved that the $(2,0)$-part of
$b$ satisfies a Chern heat equation along the unnormalized
generalized KR flow with initial data of symplectic
type~\cite[Proposition~4.9]{ASU-toric}.

The following lemma gives its normalized form and
the parabolic rescaling to the unnormalized pluriclosed system
used by Jordan--Streets~\cite{JordanStreets}.

\begin{lemma}\label{lem:JS-convention}
Along the normalized flow,
\begin{equation}\label{eq:normalized-PCF-pair}
 \partial_t\omega_J=-2(\rho_J^B)^{1,1}+2\omega_J,
 \qquad
 \partial_t\beta^{\rm JS}=2(\rho_J^B)^{2,0}+2\beta^{\rm JS}
 =\Delta_g^C\beta^{\rm JS}+2\beta^{\rm JS}.
\end{equation}
Under the change of variables
\[
 \widehat g=e^{-2t}g,\qquad
 \widehat\beta=e^{-2t}\beta^{\rm JS},\qquad
 \sigma=1-e^{-2t},
\]
{with $\widehat\omega:=e^{-2t}\omega_J$,}
the pair satisfies the unnormalized equations
\begin{equation}\label{eq:JS-component-system}
 \partial_\sigma\widehat\omega=-(\rho_{\widehat g}^B)^{1,1},
 \qquad
 \partial_\sigma\widehat\beta
 =\tfrac12\Delta_{\widehat g}^C\widehat\beta,
 \qquad
 \bar\partial\widehat\beta=\partial\widehat\omega.
\end{equation}
\end{lemma}

\begin{proof}
Differentiate \eqref{eq:JS-beta-conversion} and use
$d\rho_J^B=0$.  It follows that
$\partial_t\beta^{\rm JS}-2(\rho_J^B)^{2,0}-2\beta^{\rm JS}$
is a holomorphic two-form.  It vanishes since $M$ is Fano.
{
The heat equation in \eqref{eq:normalized-PCF-pair} is the
normalized form of the heat equation proved in
\cite[Proposition~4.9]{ASU-toric}.}
Since the Chern connection and the
Bismut--Ricci form are unchanged by constant metric rescaling,
we obtain \eqref{eq:JS-component-system} using
$\frac{d\sigma}{dt}=2e^{-2t}$.
\end{proof}

For a pair $(g,\beta)$, let $\mathcal G$ be the Hermitian metric on
$T^{1,0}M\oplus(T^{1,0}M)^*$ with components
\begin{equation}\label{eq:JS-generalized-metric-preliminary}
\begin{aligned}
 \mathcal G_{Z^i\overline{Z^j}}
 &=g_{i\bar j}+\beta_{ik}\overline{\beta_{j\ell}}g^{\bar\ell k},
 &\mathcal G_{Z^i\overline{W^j}}&=\beta_{ip}g^{\bar jp},\\
 \mathcal G_{W^i\overline{Z^j}}
 &=g^{\bar pi}\overline{\beta_{jp}},
 &\mathcal G_{W^i\overline{W^j}}&=g^{\bar ji},
\end{aligned}
\end{equation}
where $Z^i=\partial_{z^i}$ and $W^i=dz^i$.
We use the following consequence of the local estimates in
\cite[Lemma~3.7 and Theorem~4.2]{JordanStreets}.

\begin{theorem}\label{thm:JS-local}
Let $(g,\beta)$ solve \eqref{eq:JS-component-system} on a fixed
coordinate ball $\mathbb B_R$ for $0\leq\sigma<\tau\leq1$.
Suppose its generalized metric satisfies
\[
 \Gamma^{-1}\widetilde{\mathcal G}\leq\mathcal G
 \leq\Gamma\widetilde{\mathcal G}
\]
for a fixed smooth background metric $\widetilde{\mathcal G}$.
Then for each $k\geq0$ and $0<\sigma_0<\tau$, all spatial derivatives
of $g$ and $\beta$ of order at most $k$ are bounded on
$\mathbb B_{\frac{R}{2}}\times[\sigma_0,\tau)$ by a constant depending
only on $k,\Gamma,\sigma_0,R$ and the background geometry.
The bound is independent of the derivatives of the initial data.
\end{theorem}

\section{The extended Mabuchi functional}\label{sec:mabuchi}

\begingroup
In this section we introduce an extended Mabuchi functional and
prove its convexity on the smooth admissible class,
properness, and monotonicity along the flow
\eqref{eq:intro-flow}. These estimates give compactness of normalized
potentials in $L^1(P)$.

Inspired by Tian--Zhu's work~\cite{TianZhu} on the modified Mabuchi
$K$-energy and Apostolov--Streets--Ustinovskiy's work on generalized
K\"ahler metrics~\cite[Definition~5.8 and Lemma~5.10]{ASU-toric},
we introduce the following functional.

\begin{definition}\label{def:M}
For a smooth admissible pair
$(u,B)\in\mathcal G_\T(P)$, set $U:=D^2u$. An extended
Mabuchi functional on $\mathcal G_\T(P)$ is defined by
\[
 \cM_\xi(u,B)
 =2\int_Px\cdot\nabla u\,\rho_\xi\dd x
 -\int_P\log\det(U+\ii B)\,\rho_\xi\dd x.
\]
Here $\rho_\xi\dd x=e^{\langle\xi,x\rangle}\dd x$ is the
soliton-weighted Lebesgue measure fixed by \eqref{eq:barycenter}.
In particular, $\rho_0\equiv1$.
\end{definition}

When $\xi=0$, Definition~\ref{def:M} agrees, up to the normalization
in \eqref{eq:intro-flow}, with the toric formula for the extended
Mabuchi energy of Apostolov--Streets--Ustinovskiy
\cite[Definition~5.8 and Lemma~5.10]{ASU-toric}.
Restricting Definition~\ref{def:M} to $B=0$ gives
\begin{equation}\label{eq:M-B-zero-preliminary}
 \cM_\xi(u,0)
 =2\int_Px\cdot\nabla u\,\rho_\xi\dd x
 -\int_P\log\det D^2u\,\rho_\xi\dd x.
\end{equation}
This agrees, up to an additive constant, with the toric modified
$K$-energy of Wang--Zhou--Zhu evaluated at $2u$
\cite[Proposition~3.1]{WangZhouZhu}.
\endgroup

\begingroup
\begin{lemma}\label{prop:M-basic}
The functional $\cM_\xi(u,B)$ is finite and invariant under
addition of affine functions. Moreover,
       \begin{equation}\label{eq:M-comparison}
        \cM_\xi(u,B)
        =\cM_\xi(u,0)
        +\int_P\log\frac{\det U}{\det(U+\ii B)}\,
          \rho_\xi\dd x
        \geq \cM_\xi(u,0).
       \end{equation}
\end{lemma}

\begin{proof}
 Let $L_1=0,\ldots,L_r=0$ be the facets meeting along a boundary
 stratum of codimension $r$. The Guillemin model gives
 \[
  \log\det(U+\ii B)
  =-\sum_{\alpha=1}^r\log L_\alpha+O(1).
 \]
 Here $O(1)$ denotes a term bounded up to the stratum.
 Thus the logarithmic singularities are integrable and so
 $\cM_\xi(u,B)$ is finite. If $\ell=c+a\cdot x$, then using the
 weighted barycenter condition \eqref{eq:barycenter}, we obtain
 \[
  \cM_\xi(u+\ell,B)-\cM_\xi(u,B)
  =2a\cdot\int_Px\rho_\xi\dd x=0.
 \]

Recall the matrix $S=U^{-\frac12}BU^{-\frac12}$ in
\eqref{eq:skew-normal-form}. Then positivity implies $s_\alpha<1$,
and it follows that
\[
 \frac{\det(U+\ii B)}{\det U}
 =\det(I+\ii S)=\prod_\alpha(1-s_\alpha^2)\leq1,
\]
with equality if and only if $B=0$. Combining this with
\eqref{eq:M-B-zero-preliminary} establishes \eqref{eq:M-comparison}.
\end{proof}
\endgroup

\begingroup
\subsection{Monotonicity of $\cM_\xi(\cdot,\cdot)$}

In this section, we compute the variation of $\cM_\xi(\cdot,\cdot)$
and then get the monotonicity along \eqref{eq:intro-flow}.
We need to introduce some notation. Set
\begin{equation}\label{eq:matrix-notation}
 (U+\ii B)^{-1}=:X+\ii Y,\qquad
 \Lambda:=\log\det(U+\ii B),\qquad
 \Lambda_i:=\partial_i\Lambda.
\end{equation}
Thus $X=\operatorname{Re}\bigl((U+\ii B)^{-1}\bigr)$.
Moreover, we have

\begin{lemma}\label{lem:piola}
\begin{equation}\label{eq:piola}
 \partial_iX^{ij}=-X^{ij}\Lambda_i.
\end{equation}
\end{lemma}

\begin{proof}
 Since $B$ is constant, $\partial_k(U+\ii B)_{ij}=u_{ijk}$
is symmetric in all three indices.  Differentiating $(U+\ii B)^{-1}$
and contracting, we obtain
\[
 \partial_i\bigl((U+\ii B)^{-1}\bigr)^{ij}
 =-\bigl((U+\ii B)^{-1}\bigr)^{bj}
   \partial_b\log\det(U+\ii B).
\]
The determinant is positive and real by \eqref{eq:positivity}.
Taking real parts yields \eqref{eq:piola}.
\end{proof}
\endgroup

The following boundary identity is due to the compactification
criterion of Apostolov--Streets--Ustinovskiy~\cite[Lemma 6.3(iii)$'$]{ASU-space}.
For completeness, we give a direct proof below.

\begin{lemma}\label{lem:facet}
	For each facet $F=\{L_F=0\}$, let $\nu_F=\dd L_F$
	be its inward normal covector. Then the matrix-valued
	functions $X$, $Y$, and $\hat C$ extend smoothly to
	$\bar P$ and satisfy
\begin{equation}\label{eq:zero-flux}
 (\nu_F)_iX^{ij}=(\nu_F)_iY^{ij}
 =(\nu_F)_i\hat C^i{}_k=0
 \qquad\text{on }F.
\end{equation}
\end{lemma}

\begin{proof}
Let $z_0$ lie in the relative interior of a
codimension-$r$ face contained in $F_1,\ldots,F_r$.
Since $P$ is Delzant, choose affine coordinates with
$x_\alpha=L_{F_\alpha}$ for $1\leq\alpha\leq r$ and the
remaining coordinates tangent to the face.
By the Guillemin boundary condition, we have
 \[
  u=\frac12\sum_{\alpha=1}^r
      L_{F_\alpha}\log L_{F_\alpha}+h,
  \qquad h\text{ smooth near }z_0.
 \]
 Set $\hat L:=2\,\operatorname{diag}(L_{F_1},\ldots,L_{F_r})$.
 With respect to the normal--tangential splitting, write
 \[
  U+\ii B=
  \begin{pmatrix}
   \hat L^{-1}+R_0&Q^*\\
   Q&H
  \end{pmatrix}.
 \]
 The blocks $R_0,Q,H$ are smooth. For $r<m$,
 facewise positivity makes $H$ positive definite near $z_0$, so
 $H^{-1}$ is smooth there.
\begingroup
For $r<m$, put $R=R_0-Q^*H^{-1}Q$.
When $r=m$, the blocks $Q,H$ are absent and we put $R=R_0$.
In both cases, set $\hat S=\hat L^{-1}+R$.  Then
\endgroup
 \begin{equation}\label{eq:corner-Schur-factorization}
  \hat S^{-1}=(I+\hat L R)^{-1}\hat L
  =\hat L(I+R\hat L)^{-1},
 \end{equation}
 while for $r<m$ direct block inversion gives
 \[
  (U+\ii B)^{-1}
  =\begin{pmatrix}
    \hat S^{-1}
      &-\hat S^{-1}Q^*H^{-1}\\
    -H^{-1}Q\hat S^{-1}
      &H^{-1}+H^{-1}Q\hat S^{-1}Q^*H^{-1}
   \end{pmatrix}.
 \]
 For $r=m$, one simply has $(U+\ii B)^{-1}=\hat S^{-1}$.
 Since $\hat L=0$ at $z_0$, the inverses on the right-hand side of
 \eqref{eq:corner-Schur-factorization} are smooth near $z_0$.
 The block formula shows that $(U+\ii B)^{-1}$ extends smoothly across
 the face.
\begingroup
For $1\leq\alpha\leq r$, the factorization with $\hat L$ on the
right shows that the $\alpha$th column of $\hat S^{-1}$ is
divisible by $L_{F_\alpha}$; the factorization with $\hat L$ on
the left gives the same property for the $\alpha$th row.
By the block inverse formula, these divisibilities also hold for
$(U+\ii B)^{-1}$. Hence its $\alpha$th row and column
vanish on $F_\alpha$.

In these coordinates, $\nu_{F_\alpha}=\dd x_\alpha$.
Taking real and imaginary parts and using $\hat C=YB$, we obtain
the smooth extensions and \eqref{eq:zero-flux}.
\endgroup

\begingroup
Choose a Euclidean metric with volume form $\dd x$, and identify
vectors and covectors using this metric. Let $n_F$ and $\dd S$
be the outward unit normal and the induced surface measure on $F$.
Set
\[
 \dd\sigma_F:=\frac{\dd S}{|\nu_F|},
 \qquad n_F\,\dd S=-\nu_F\,\dd\sigma_F.
\]
By \eqref{eq:zero-flux}, the flux vanishes on each facet, so every
smooth map $f:\bar P\to\mathbb R^m$ satisfies
\[
 \sum_F\int_F\rho_\xi(\hat C f)\cdot n_F\,\dd S=0.
\]
\endgroup
\end{proof}

\begingroup
We now compute the first variation of $\cM_\xi(u_t,B_t)$.
Let $v=\delta u\in C^\infty(\bar P)$ and keep $B$ fixed.  Then
\[
 D_u\cM_\xi(u,B)[v]
 =2\int_P\rho_\xi x\cdot\nabla v\dd x
 -\int_P\rho_\xi X^{ij}v_{ij}\dd x.
\]
Set $p=\dot u-\langle\xi,x\rangle$.
Taking the gradient of \eqref{eq:intro-flow} and using
$\nabla(x\cdot\nabla u)=\nabla u+Ux$, we find
\[
 \nabla\dot u=\nabla\Lambda-2Ux.
\]
Since $\nabla p=\nabla\dot u-\xi$, it follows that
\begin{equation}\label{eq:gradL}
 \nabla\Lambda-\xi=\nabla p+2Ux.
\end{equation}

By Lemma~\ref{lem:piola},
\[
 \partial_i(\rho_\xi X^{ij})
 =\rho_\xi X^{ij}(\xi_i-\Lambda_i).
\]
Integrating by parts over
$P_\varepsilon=\{L_r\geq\varepsilon\text{ for all }r\}$ gives
\[
 D_u\cM_\xi(u,B)[v]
 =-\int_P\rho_\xi
   \{X(\nabla\Lambda-\xi)-2x\}\cdot\nabla v\dd x
\]
after letting $\varepsilon\to0^+$.
The boundary term tends to zero by \eqref{eq:zero-flux} and
$n_F\,\dd S=-\nu_F\,\dd\sigma_F$, while the corner estimates
in the proof of Lemma~\ref{lem:facet} justify passage to the limit.
Substituting \eqref{eq:gradL} yields
\begin{equation}\label{eq:u-variation}
 D_u\cM_\xi(u,B)[v]
 =-\int_P\rho_\xi\{X\nabla p+2(XU-I)x\}\cdot\nabla v\dd x.
\end{equation}
\endgroup

The $B$-variation is given by
\begin{equation}\label{eq:B-variation}
 D_B\cM_\xi(u,B)[\dot B]
 =\int_P\rho_\xi\tr(Y\dot B)\dd x.
\end{equation}
Indeed,
\[
 \delta_B\log\det(U+\ii B)
 =\tr\bigl((X+\ii Y)\ii\dot B\bigr)
 =-\tr(Y\dot B),
\]
because the contraction of the symmetric matrix $X$ with the
skew-symmetric matrix $\dot B$ vanishes.

Since $\dot u-p=\langle\xi,x\rangle$ is affine,
affine invariance allows us to replace $\dot u$ by $p$ in
\eqref{eq:u-variation}.
Using $\dot B=-2B$ and \eqref{eq:B-variation}, we obtain
\begin{equation}\label{eq:variation-before-square}
\begin{split}
 \frac{\dd}{\dd t}\cM_\xi(u_t,B_t)
 ={}&-\int_P\rho_\xi X(\nabla p,\nabla p)\dd x\\
 &-2\int_P\rho_\xi(\hat Cx)\cdot\nabla p\dd x
 -2\int_P\rho_\xi\tr \hat C\dd x.
\end{split}
\end{equation}

The second integral on the right-hand side of
\eqref{eq:variation-before-square} is handled by the following
weighted divergence identity.

\begin{lemma}\label{lem:C-div}
 Along \eqref{eq:intro-flow},
\begin{equation}\label{eq:C-div}
 \partial_i(\rho_\xi \hat C^i{}_k)
 =-\rho_\xi \hat C^i{}_k(p_i+2u_{ij}x_j).
\end{equation}
 Consequently,
\begin{equation}\label{eq:mixed-global}
 \int_P\rho_\xi(\hat Cx)\cdot\nabla p\dd x
 =\int_P\rho_\xi\tr \hat C\dd x
 -2\int_P\rho_\xi x^TU\hat Cx\dd x.
\end{equation}
\end{lemma}

\begin{proof}
 Differentiating $\hat C=XU-I$, applying Lemma~\ref{lem:piola}, and using
 the full symmetry of $u_{ijk}$ gives
 $\partial_i\hat C^i{}_k
=-\hat C^i{}_k\Lambda_i$.
 Since $\partial_i \rho_\xi=\xi_i \rho_\xi$,
 equation~\eqref{eq:C-div} follows from \eqref{eq:gradL}.
 Multiplying by $x_k$, integrating over $P$, and
 using \eqref{eq:zero-flux}, we obtain \eqref{eq:mixed-global}.
\end{proof}

By \eqref{eq:mixed-global} together with
Lemma~\ref{lem:C-div}, we prove

\begin{theorem}\label{thm:dissipation}
 Along \eqref{eq:intro-flow},
\begin{equation}\label{eq:exact-dissipation}
\begin{split}
 \frac{\dd}{\dd t}\cM_\xi(u_t,B_t)
 ={}&-\int_P\rho_\xi U^{-1}(\nabla p,\nabla p)\dd x\\
 &-\int_P\rho_\xi(X-U^{-1})
   (\nabla p+2Ux,\nabla p+2Ux)\dd x
 \leq0.
\end{split}
\end{equation}
\end{theorem}

\begin{proof}
 By \eqref{eq:D-positive},
\[
\begin{split}
 &\int_P\rho_\xi U^{-1}(\nabla p,\nabla p)\dd x
 +\int_P\rho_\xi (X-U^{-1})(\nabla p+2Ux,\nabla p+2Ux)\dd x\\
 &\quad=
 \int_P\rho_\xi X(\nabla p,\nabla p)\dd x
 +4\int_P\rho_\xi(\hat Cx)\cdot\nabla p\dd x
 +4\int_P\rho_\xi x^TU\hat Cx\dd x.
\end{split}
\]
Equation~\eqref{eq:mixed-global} converts the last two terms into
\[
 2\int_P\rho_\xi(\hat Cx)\cdot\nabla p\dd x
 +2\int_P\rho_\xi\tr \hat C\dd x.
\]
Comparison with \eqref{eq:variation-before-square} proves the
identity.  Nonpositivity follows from \eqref{eq:D-positive}.
\end{proof}

\begin{remark}\label{rem:integrability}
 The two densities in \eqref{eq:exact-dissipation} are bounded near
 $\partial P$.  At a codimension-$r$ corner,
 \[
  \Lambda=-\sum_{\alpha=1}^r\log L_{F_\alpha}+h,
  \qquad h\text{ smooth}.
 \]
 Equation \eqref{eq:corner-Schur-factorization} and its $B=0$
 specialization show that, for
 \[
  \eta:=\nabla\Lambda-\xi=\nabla p+2Ux,
 \]
 both $U^{-1}\eta$ and $X\eta$ are bounded.  Since
 \[
  X-U^{-1}=-XBU^{-1}BU^{-1},
 \]
 it follows that
 \[
  (X-U^{-1})(\eta,\eta)
  =-(X\eta)^TB U^{-1}B(U^{-1}\eta)=O(1).
 \]
 The function $p=\dot u-\langle\xi,x\rangle$ is smooth on $\bar P$,
 so $U^{-1}(\nabla p,\nabla p)=O(1)$ as well.
\end{remark}

\begingroup
We now turn to  the study of the convexity and properness of $\cM_\xi$.
Let $u_{\KRS}$ be the symplectic potential of a KR soliton with
vector $\xi$. Since $\cM_\xi(u,0)$ is the toric modified
$K$-energy, $u_{\KRS}$ is a critical point of $\cM_\xi(u,0)$
\cite[Section~5]{TianZhu}. By the convexity of $\cM_\xi(u,0)$
and \eqref{eq:M-comparison}, we obtain
\begin{equation}\label{eq:M-minimum}
 \cM_\xi(u,B)\geq\cM_\xi(u_{\KRS},0).
\end{equation}

Furthermore, we have

\begin{proposition}\label{prop:M-variational-structure}
The functional $\cM_\xi$ has the following properties.
\begin{enumerate}[label=\textup{\roman*)},leftmargin=*,itemsep=0.4\baselineskip]
\item $\cM_\xi(\cdot,\cdot)$ is jointly convex on $\mathcal G_\T(P)$
and is strictly convex modulo addition of affine functions to $u$.
In particular, $(u_{\KRS},0)$ is the unique critical point and
minimizer of $\cM_\xi(\cdot,\cdot)$ modulo affine functions.
Moreover, equality holds in \eqref{eq:M-minimum} if and only if
$B=0$ and $u-u_{\KRS}$ is affine.

\item There are constants $\delta>0$ and $C$ such that for any
$(u,B)\in\mathcal G_\T(P)$,
\begin{equation}\label{eq:generalized-relative-properness}
 \cM_\xi(u,B)\geq\delta\int_P\widehat u\,\rho_\xi\dd x-C,
\end{equation}
where
\begin{equation}\label{normalized-u}
 \widehat u(x):=u(x)-u(0)-\langle\nabla u(0),x\rangle.
\end{equation}
\end{enumerate}
\end{proposition}

\begin{proof}
\begingroup
i) Let
\[
 (u_s,B_s)=(1-s)(u_0,B_0)+s(u_1,B_1),\qquad 0\leq s\leq1,
\]
be an affine segment in $\mathcal G_\T(P)$. Set
\[
 K_s=D^2u_s+\ii B_s,\qquad
 \dot K=\frac{\dd K_s}{\dd s}
 =D^2(u_1-u_0)+\ii(B_1-B_0).
\]
Then
$\frac{\dd}{\dd s}\log\det K_s=\tr(K_s^{-1}\dot K)$.
Thus the first variation along this segment is
\[
 \frac{\dd}{\dd s}\cM_\xi(u_s,B_s)
 =2\int_P\rho_\xi x\cdot\nabla(u_1-u_0)\dd x
 -\int_P\rho_\xi\tr(K_s^{-1}\dot K)\dd x.
\]
The first integral and $\dot K$ are independent of $s$.
Differentiating $K_s^{-1}K_s=I$ yields
\[
 \frac{\dd K_s^{-1}}{\dd s}K_s+K_s^{-1}\dot K=0,
 \qquad
 \frac{\dd K_s^{-1}}{\dd s}
 =-K_s^{-1}\dot K K_s^{-1}.
\]
Differentiating once more and using cyclicity of the trace, we obtain
\begin{equation}\label{eq:M-joint-convexity}
\begin{aligned}
 \frac{\dd^2}{\dd s^2}\cM_\xi(u_s,B_s)
 &=-\int_P\rho_\xi
   \tr\!\left(\frac{\dd K_s^{-1}}{\dd s}\dot K\right)\dd x\\
 &=\int_P\rho_\xi
   \tr(K_s^{-1}\dot K K_s^{-1}\dot K)\dd x\\
 &=\int_P\rho_\xi
   \tr\!\left((K_s^{-\frac12}\dot K K_s^{-\frac12})^2\right)\dd x\\
 &=\int_P\rho_\xi
   \left|K_s^{-\frac12}\dot K K_s^{-\frac12}\right|_{\rm HS}^2\dd x
   \geq0.
\end{aligned}
\end{equation}
Here $|W|_{\rm HS}^2:=\tr(W^*W)$, and the last equality holds
because $K_s^{-\frac12}\dot K K_s^{-\frac12}$ is Hermitian.
The smooth extension of $K_s^{-1}$ in Lemma~\ref{lem:facet},
together with $u_1-u_0\in C^\infty(\bar P)$, justifies
differentiation under the integral. Since $\rho_\xi>0$, equality
holds if and only if $\dot K=0$, or equivalently $B_1=B_0$ and
$u_1-u_0$ is affine. This proves joint convexity and strict
convexity modulo affine functions.
\endgroup

At $B=0$, the derivative of $\cM_\xi(u,B)$ in a $B$-direction
$C\in\wedge^2\mathfrak t$ also vanishes:
\[
 \left.\frac{\dd}{\dd s}\right|_{s=0}\cM_\xi(u,sC)
 =-\int_P\rho_\xi\tr(U^{-1}\ii C)\dd x=0,
\]
since $U^{-1}$ is symmetric and $C$ is skew-symmetric.
Thus $(u_{\KRS},0)$ is a critical point of $\cM_\xi(u,B)$
in $\mathcal G_\T(P)$, and is also a minimizer by
\eqref{eq:M-minimum}. Uniqueness follows from the equality condition
in \eqref{eq:M-joint-convexity}.

ii) It was proved in \cite[Proposition~3.2]{WangZhouZhu} that,
under the normalization \eqref{normalized-u},
\[
 \cM_\xi(u,0)\geq\delta\int_P\widehat u\,\rho_\xi\dd x-C.
\]
Combining this with \eqref{eq:M-comparison} proves
\eqref{eq:generalized-relative-properness}.
\end{proof}
\endgroup

The lower bound \eqref{eq:M-minimum} for $\cM_\xi$
yields the following finite total dissipation.

\begin{corollary}\label{cor:finite-D}
\begingroup
Set
\[
\begin{aligned}
 \cD(t):={}&\int_P\rho_\xi U^{-1}(\nabla p,\nabla p)\dd x\\
 &+\int_P\rho_\xi(X-U^{-1})
   (\nabla p+2Ux,\nabla p+2Ux)\dd x,
\end{aligned}
\]
where $U=D^2u_t$, $X=\operatorname{Re}((U+\ii B_t)^{-1})$, and
$p=\dot u_t-\langle\xi,x\rangle$.
\endgroup
Then
\begin{equation}\label{eq:finite-D}
 \int_0^\infty\cD(t)\dd t<\infty,
 \qquad
 \int_T^\infty\cD(t)\dd t\longrightarrow0
 \quad(T\to\infty).
\end{equation}
\end{corollary}

\begingroup
\begin{proof}
By the properness estimate \eqref{eq:generalized-relative-properness},
$\cM_\xi(u_t,B_t)\geq-C$.
By definition of $\cD(t)$, the identity
\eqref{eq:exact-dissipation} becomes
$\frac{\dd}{\dd t}\cM_\xi(u_t,B_t)=-\cD(t)$.
Integrating from $0$ to $T$, we obtain
\[
 \int_0^T\cD(t)\dd t
 =\cM_\xi(u_0,B_0)-\cM_\xi(u_T,B_T)
 \leq\cM_\xi(u_0,B_0)+C.
\]
Since $\cD(t)\geq0$, letting $T\to\infty$ proves that
$\cD\in L^1([0,\infty))$. The second assertion in
\eqref{eq:finite-D} is then the vanishing of the tails of a
nonnegative integrable function.
\end{proof}
\endgroup

\begingroup
By Proposition~\ref{prop:M-variational-structure} and
Theorem~\ref{thm:dissipation}, we obtain compactness results for
$u_t$ in $L^1(P)$, and hence for $\varphi_t$ in
$\mathcal E^1$; see \cite{BB,ToricPluripotential,LTZ}.
For the reader's convenience, we give a self-contained proof below,
following the arguments in
\cite[Proposition~3.2 and Lemma~3.3]{WangZhouZhu}
and \cite[Section~4]{ZhouZhuWeak}.
\endgroup

\subsection{Compactness of $\varphi_t$}
For the reference form $\omega_0$ fixed above, set
$V=\int_M\omega_0^m$.
In our convention,
\[
 \operatorname{PSH}(M,\omega_0)
 :=\{\varphi\in L^1(M,\omega_0^m):
       \varphi\text{ is upper semicontinuous},\quad
       \omega_0+\ddcJ\varphi\geq0\}.
\]
For such a potential, set $\varphi_\ell=\max\{\varphi,-\ell\}$ and
define its non-pluripolar Monge--Amp\`ere measure by
\[
 \operatorname{MA}_{\omega_0}(\varphi)
 :=\lim_{\ell\to\infty}
    \mathbf1_{\{\varphi>-\ell\}}
    (\omega_0+\ddcJ\varphi_\ell)^m.
\]
The finite-energy space $\mathcal E^1$ is defined by
\cite{GuedjZeriahiEnergy,BBEGZ},
\[
 \mathcal E^1=\mathcal E^1(M,\omega_0)
 =\left\{\varphi\in\operatorname{PSH}(M,\omega_0):
     \int_M\operatorname{MA}_{\omega_0}(\varphi)=V,\quad
     \int_M|\varphi|\,\operatorname{MA}_{\omega_0}(\varphi)<\infty
     \right\}.
\]
Recall that the Aubin--Mabuchi energy is given by
\cite{MabuchiKenergy,Ti87},
\[
 E_{\rm AM}(\varphi)
 =\frac{1}{(m+1)V}\sum_{r=0}^m
   \int_M\varphi\,\omega_\varphi^r\wedge\omega_0^{m-r}.
\]
Here $\omega_\varphi=\omega_0+\ddcJ\varphi$.
In particular, $E_{\rm AM}(\varphi+c)=E_{\rm AM}(\varphi)+c$.
This formula is also well-defined for bounded
$\omega_0$-plurisubharmonic potentials, and for general potentials
we set
$E_{\rm AM}(\varphi)=\lim_{\ell\to\infty}E_{\rm AM}(\varphi_\ell)$.
Then $\varphi\in\mathcal E^1$ if and only if
$E_{\rm AM}(\varphi)>-\infty$.
In our case of toric manifolds, denote 
\[
 \mathcal E^1_\T
 =\{\varphi\in\mathcal E^1:
        \varphi\circ\tau=\varphi\text{ for every }\tau\in\T\}.
\]
We write $C^\infty(M)^\T$ for the space of smooth real-valued
functions $f$ on $M$ satisfying $f\circ\tau=f$ for every $\tau\in\T$.
In particular,
\[
 \mathcal H_\T
 :=\{\varphi\in C^\infty(M)^\T:\omega_\varphi>0\}
 \subset\mathcal E^1_\T\subset\mathcal E^1,
 \qquad 0\in\mathcal H_\T.
\]
Moreover, Legendre duality gives the following characterization
of the $\mathcal E^1_\T$-space
\cite[Proposition~3.9]{ToricPluripotential}
(also see \cite{BB,LTZ}),
\begin{equation}\label{eq:toric-finite-energy}
 \varphi\in\mathcal E^1_\T
 \quad\Longleftrightarrow\quad
 u_\varphi:=(\bar\phi_0+\varphi)^*\in L^1(P),
\end{equation}
where the Legendre transform
$u_\varphi(x)=\sup_{y\in\mathbb R^m}
\{\langle x,y\rangle-(\bar\phi_0+\varphi)(y)\}$
is convex and finite on $P^\circ$ when
$\varphi\in\mathcal E^1_\T$.
In fact, we have
\begin{equation}\label{eq:toric-AM-energy}
 E_{\rm AM}(\varphi)
 =-\frac{1}{\Vol(P)}\int_P(u_\varphi-\bar u_0)\dd x,
\end{equation}
where $\bar u_0:=\bar\phi_0^*$.
For a smooth potential, this follows by differentiating:
$\dot\varphi=-\dot u_\varphi\circ\nabla(\bar\phi_0+\varphi)$
and the moment map pushes $\frac{\omega_\varphi^m}{V}$ forward to
$\frac{\dd x}{\Vol(P)}$.
By monotone approximation, \eqref{eq:toric-AM-energy} extends
to $\mathcal E^1_\T$.

By \eqref{eq:toric-AM-energy}, we conclude that
\begin{equation}\label{eq:Legendre-L1-convergence}
 u_{\varphi_j}\longrightarrow u_\varphi\ \text{in }L^1(P)
 \quad\Longrightarrow\quad
 \begin{cases}
 \varphi_j\longrightarrow\varphi
       &\text{in }L^1(M,\omega_0^m),\\
 E_{\rm AM}(\varphi_j)\longrightarrow E_{\rm AM}(\varphi).
 \end{cases}
\end{equation}
Indeed, convergence of convex functions in $L^1(P)$ gives locally
uniform convergence on $P^\circ$ and, by Legendre duality, locally
uniform convergence of their conjugates on $\mathbb R^m$.
These conjugates have a common Lipschitz constant, since their
slopes lie in $P$. Their values
$u_{\varphi_j}^*(0)=-\inf_Pu_{\varphi_j}$ stay bounded, so
$|u_{\varphi_j}^*(y)|\leq C(1+|y|)$ uniformly in $j$.
The reference volume in logarithmic coordinates has finite first
moment: this is equivalent to
$\int_P|\nabla\bar u_0|\dd x<\infty$, which follows from the
Guillemin asymptotics. Thus dominated convergence proves the first
assertion, and \eqref{eq:toric-AM-energy} implies the second.

\begingroup
\begin{remark}\label{rem:finite-energy-convexity}
On $\mathcal E^1_\T$, the convexity
of the $K$-energy is understood in Legendre
variables: $u_s=(1-s)u_0+su_1$, or equivalently along toric weak
geodesics of K\"ahler potentials. Zhou--Zhu establish
lower-semicontinuity for weak toric energies in a class with
boundary control~\cite[Proposition~3.4 and Remark~3.5]{ZhouZhuWeak}.
The ordinary Mabuchi energy admits a lower-semicontinuous
extension to $\mathcal E^1$ that is convex along finite-energy
geodesics~\cite[Theorem~1.2]{BDL-convexity}. The corresponding
result for weighted Mabuchi energies on $\mathcal E^1_\T$ is
proved in \cite[Theorem~6.1]{AJL-weighted}.

The additional parameter $B$ in $\cM_\xi$ requires a separate
argument. We define the following extended-valued,
lower-semicontinuous extension by smooth admissible approximation:
\[
 \overline{\cM}_\xi(u,B)
 :=\inf\left\{\liminf_{j\to\infty}\cM_\xi(u_j,B_j):
 \begin{array}{l}
  (u_j,B_j)\in\mathcal G_\T(P),\\
  u_j\to u\text{ in }L^1(P),\quad B_j\to B
 \end{array}\right\},
\]
with the infimum of the empty set equal to $+\infty$.

To prove the convexity, suppose first that both endpoint values are
finite. By a diagonal argument, choose smooth admissible pairs
$(u_{i,j},B_{i,j})$ such that $u_{i,j}\to u_i$ in $L^1(P)$,
$B_{i,j}\to B_i$, and
$\cM_\xi(u_{i,j},B_{i,j})\to\overline{\cM}_\xi(u_i,B_i)$,
$i=0,1$. Their convex combinations remain admissible and converge
to the corresponding combination of the endpoints. Applying
Proposition~\ref{prop:M-variational-structure} and taking lower
limits in the definition of $\overline{\cM}_\xi$ yields, for $0<s<1$,
\[
 \overline{\cM}_\xi(u_s,(1-s)B_0+sB_1)
 \leq(1-s)\overline{\cM}_\xi(u_0,B_0)
      +s\overline{\cM}_\xi(u_1,B_1).
\]
If either endpoint value is $+\infty$, this inequality is automatic.

The lower bound \eqref{eq:M-minimum} passes to every approximating
sequence and is attained by the constant sequence $(u_{\KRS},0)$,
so the soliton remains a minimizer. For the ordinary extended
$K$-energy, regularity of $\mathcal E^1$ minimizers and uniqueness
modulo automorphisms are proved under the assumption that a smooth
cscK metric exists~\cite[Theorem~1.4]{BDL-minimizers}.
In our argument, the uniqueness is used only after the limiting soliton
is known to be smooth, by Proposition~\ref{prop:M-variational-structure}.
\end{remark}
\endgroup

For $\tau\in\TC$, denote by $\tau[\varphi]$ the potential of
$\tau^*\omega_\varphi$ normalized by $E_{\rm AM}(\tau[\varphi])=0$.
The above transformation can be explicitly described in
logarithmic coordinates.
Write $\tau=k\exp(a)$, where $k\in\T$ and $\exp(a)$ acts by
$y\mapsto y+a$. Taking Legendre transforms, we obtain
\begin{equation}\label{eq:torus-action-Legendre}
 u_{\tau[\varphi]}(x)=u_\varphi(x)-\langle a,x\rangle+c.
\end{equation}
Hence the $\TC$-action adds an affine function to $u$ and leaves
$D^2u$ unchanged. By Lemma~\ref{prop:M-basic}, this also verifies
that $\cM_\xi$ is invariant under the induced $\TC$-action.

\begin{proposition}\label{prop:toric-compactness}
 For every sequence $t_j\to\infty$, there are $\tau_j\in\TC$ and a
 subsequence such that for
 $\widehat\varphi_j:=\tau_j[\varphi_{t_j}]$, 
 \begin{equation}\label{eq:toric-compactness}
  u_{\widehat\varphi_j}\longrightarrow u_\infty
       \quad\text{in }L^1(P),\qquad
  \widehat\varphi_j\longrightarrow\varphi_\infty
       \quad\text{in }L^1(M,\omega_0^m).
 \end{equation}
 Here $u_\infty$ is convex and integrable on $P$,
 $\varphi_\infty=u_\infty^*-\bar\phi_0\in\mathcal E^1_\T$, and
 $E_{\rm AM}(\varphi_\infty)=0$.
\end{proposition}

\begin{proof}
Let $\widehat u_t$ be the normalization
\eqref{normalized-u}. Then $\widehat u_t\geq0$ and
$\widehat u_t(0)=0$.
 By the monotonicity of $\cM_\xi(u_t,B_t)$ in
 Theorem~\ref{thm:dissipation} and
 \eqref{eq:generalized-relative-properness}, we have
 \begin{equation}\label{eq:convex-flow-bound}
  \int_P\widehat u_t\rho_\xi\dd x\leq C.
 \end{equation}
\begingroup
Fix a Euclidean metric with volume form $\dd x$. On a facet $F$,
let $n_F$ and $\dd S$ be the outward unit normal and induced measure,
and put $h_F=x\cdot n_F>0$.
Each $\widehat u_t$ has finite boundary integral, so the logarithmic
determinant estimate applies
\cite[Lemma~3.3 and Remark~3.1]{WangZhouZhu}.
Its interior term is bounded by \eqref{eq:convex-flow-bound}, so
\endgroup
 \[
  \int_P\log\det D^2\widehat u_t\,\rho_\xi\dd x
  \leq\sum_Fh_F\int_F\widehat u_t\rho_\xi\,\dd S+C.
 \]
 Integration by parts gives
 \[
  \sum_F h_F\int_F\widehat u_t\rho_\xi\,\dd S
  =\int_Px\cdot\nabla\widehat u_t\,\rho_\xi\dd x
   +\int_P(m+\langle\xi,x\rangle)\widehat u_t\rho_\xi\dd x.
 \]
 The last two displays and \eqref{eq:convex-flow-bound} imply
 \[
  \sum_Fh_F\int_F\widehat u_t\rho_\xi\,\dd S
  \leq\cM_\xi(u_t,0)+C
  \leq\cM_\xi(u_0,B_0)+C,
 \]
 where the second inequality uses \eqref{eq:M-comparison} and
 Theorem~\ref{thm:dissipation}.  Since $\rho_\xi$ and the $h_F$ are positive,
 $\int_{\partial P}\widehat u_t\,\dd S\leq C$.

 Since $\rho_\xi$ is bounded below by a positive constant,
 we obtain local uniform bounds for $\widehat u_t$ on $P^\circ$
 from \eqref{eq:convex-flow-bound} and Jensen's inequality.
 By convexity, these bounds imply uniform Lipschitz
 estimates on compact subsets of $P^\circ$.  By the Arzel\`a--Ascoli theorem,
 after passing to
 a subsequence, $\widehat u_{t_j}$ converges locally uniformly to
 a nonnegative convex function $\widehat u_\infty$.
 To obtain convergence on all of $P$, decompose $P$ into the cones
 joining $0$ to its facets.  On the cone over $F$, write $x=ry$,
 where $y\in F$ and $0<r<1$; then
 $\dd x=h_Fr^{m-1}\dd r\,\dd S(y)$ and
 $\widehat u_{t_j}(ry)\leq r\widehat u_{t_j}(y)$.
 Consequently, for $0<a<1$,
 \begin{equation}\label{eq:convex-boundary-layer}
  \int_{P\setminus aP}\widehat u_{t_j}\dd x
  \leq\sum_F h_F\int_a^1r^m\dd r
             \int_F\widehat u_{t_j}\dd S
  \leq C(1-a).
 \end{equation}
 By Fatou's lemma, the same bound holds for
 $\widehat u_\infty$.
 Combining local uniform convergence with
 \eqref{eq:convex-boundary-layer}, we conclude that
 $\widehat u_{t_j}\to\widehat u_\infty$ in $L^1(P)$.

 Choose $\tau_j$ with translation part $\nabla u_{t_j}(0)$.
 By \eqref{eq:torus-action-Legendre} and the energy normalization,
 \[
  u_{\widehat\varphi_j}=\widehat u_{t_j}+c_j,\qquad
  c_j=\frac{1}{\Vol(P)}\int_P(\bar u_0-\widehat u_{t_j})\dd x.
 \]
 The constants $c_j$ converge to a constant $c_\infty$.
 Set $u_\infty:=\widehat u_\infty+c_\infty$; then the first convergence
 in \eqref{eq:toric-compactness} follows.  The second convergence and
 the assertions about $\varphi_\infty$ now follow from
 \eqref{eq:toric-finite-energy}--\eqref{eq:Legendre-L1-convergence}.
\end{proof}

\section{Proof of Theorem \ref{thm:conditional-weak-B-flow}}

\begingroup
In this section we prove Theorem~\ref{thm:conditional-weak-B-flow}.
We will first recall the work of Di Nezza--Lu on weak solutions
of the KR flow
\cite[Theorem~B and Theorem~5.4]{DNL}.
To express their equation in our convention, write the potential
as $v=2\varphi$, so that
$\omega+\frac12\ddcJ v=\omega+\ddcJ\varphi$.
Set $c(v):=\sup\{\lambda>0:e^{-2\lambda v}\in L^1(M,\omega^m)\}$.
Since every $\varphi\in\mathcal E^1$ has zero Lelong numbers,
$c(2\varphi)=+\infty$~\cite{GuedjZeriahiEnergy}.

The following theorem was proved in
\cite[Theorem~B and Theorem~5.4]{DNL}.
\endgroup

\begin{theorem}
\label{thm:DNL-original}
Let $(M,\omega)$ be compact K\"ahler, and set
\[
 \theta_\sigma=\omega-\sigma\Ric(\omega),\qquad
 T_{\max}=\sup\{\sigma\geq0:[\theta_\sigma]\text{ is nef}\}.
\]
Suppose that $\frac{v_0}{2}\in
\operatorname{PSH}(M,\omega)$ and $c(v_0)=+\infty$.
Then the equation
\begin{equation}\label{eq:DNL-unnormalized}
 \partial_\sigma v
 =\log\frac{(\theta_\sigma+\tfrac12\ddcJ v)^m}{\omega^m},
 \qquad
 v(\sigma)\longrightarrow v_0\quad\text{in }L^1(M,\omega^m)
 \quad(\sigma\to0^+)
\end{equation}
has a unique weak solution.
It is smooth and $\theta_\sigma+\tfrac12\ddcJ v(\sigma)>0$
for $0<\sigma<T_{\max}$.
If $v_{0,j}$ are smooth potentials with
$\omega+\tfrac12\ddcJ v_{0,j}>0$ and
$v_{0,j}\to v_0$ in $L^1(M,\omega^m)$, then the corresponding
smooth solutions $v_j$ of the evolution equation in
\eqref{eq:DNL-unnormalized}, with initial values $v_j(0)=v_{0,j}$,
satisfy
\[
 v_j(\sigma)\longrightarrow v(\sigma)
 \quad\text{in }C^\infty(M)
 \quad\text{for each fixed }\sigma\in(0,T_{\max}).
\]
\end{theorem}

\begingroup
Guedj--Zeriahi~\cite{GuedjZeriahiKRF} established the existence
and positive-time regularity of the weak flow.
Building on their work, Di Nezza--Lu proved uniqueness and
stability~\cite[Theorems~5.3 and~5.4]{DNL}, yielding
Theorem~\ref{thm:DNL-original}.
The uniform parabolic estimates used in the proof of
\cite[Theorem~5.4]{DNL}, together with convergence at each fixed
positive time, imply $C^\infty$ convergence on compact positive-time
cylinders. We use this form of convergence in the following
normalized version.
\endgroup

\begingroup
\begin{corollary}\label{cor:DNL-normalized}
Let $[\omega]=2\pi c_1(M,J)$ and let $h_\omega$ satisfy
$\ddcJ h_\omega=\Ric(\omega)-\omega$ and
$\int_M h_\omega\,\omega^m=0$.
Let $\psi_j(0)$ be smooth potentials with $\omega+\ddcJ\psi_j(0)>0$.
Suppose that $2\psi_j(0)\to v_0$ in $L^1(M,\omega^m)$, where
$\frac{v_0}{2}\in\operatorname{PSH}(M,\omega)$ and $c(v_0)=+\infty$.
Then the smooth solutions $\psi_j$ of
\begin{equation}\label{eq:normalized-KRF}
 \partial_s\psi
 =\log\frac{(\omega+\ddcJ\psi)^m}{\omega^m}
   +2\psi-2h_\omega
\end{equation}
converge to the unique weak solution $\psi$ of
\eqref{eq:normalized-KRF} with
$\psi(s)\to\frac{v_0}{2}$ in $L^1(M,\omega^m)$ as $s\to0^+$.
More precisely, $\psi_j\to\psi$ in
$C^\infty([\varepsilon,S]\times M)$ for every
$0<\varepsilon<S<\infty$.
In particular,
\[
 \|\psi_j\|_{C^r([\varepsilon,S]\times M)}
 \leq C_{\varepsilon,S,r}\qquad(r\geq0),
\]
where $C_{\varepsilon,S,r}$ is independent of $j$.
On a spacetime cylinder, $C^r$ denotes the norm of mixed
space--time derivatives.
\end{corollary}
\endgroup

\begin{proof}
Set
\[
 \sigma=1-e^{-2s},\qquad
 \theta_\sigma=\omega-\sigma\Ric(\omega),\qquad
 c_m(\sigma)=-m\{(1-\sigma)\log(1-\sigma)+\sigma\},
\]
and define
\begin{equation}\label{eq:DNL-exact-potential-change}
 v_j(\sigma)=2(1-\sigma)\psi_j(s(\sigma))
              +2\sigma h_\omega+c_m(\sigma).
\end{equation}
Then
\[
 \theta_\sigma+\tfrac12\ddcJ v_j
 =(1-\sigma)(\omega+\ddcJ\psi_j),\qquad
 \partial_\sigma v_j
 =\log\frac{(\theta_\sigma+\tfrac12\ddcJ v_j)^m}{\omega^m}.
\]
{Thus $v_j$ solves \eqref{eq:DNL-unnormalized} with initial
value $2\psi_j(0)$ and cohomological maximal time $1$.}
The uniqueness and smooth convergence at each fixed
positive time follow from Theorem~\ref{thm:DNL-original}.

\begingroup
For each compact interval $I\Subset(0,1)$, we obtain uniform
spatial derivative bounds from the estimates of
Guedj--Zeriahi~\cite{GuedjZeriahiKRF}, used by
Di Nezza--Lu~\cite[proof of Theorem~5.4]{DNL};
the time derivatives are then controlled by the equation.  Thus
\[
 \|v_j\|_{C^r(I\times M)}\leq C_{I,r}\qquad(r\geq0).
\]
\endgroup
These bounds ensure smooth precompactness on $I\times M$.
The fixed-time convergence identifies every subsequential limit with $v$,
so the whole sequence converges. The assertion follows after the
change of variables \eqref{eq:DNL-exact-potential-change}.
\end{proof}

\begingroup
In the rest of this section, we use
Corollary~\ref{cor:DNL-normalized} to prove
Theorem~\ref{thm:conditional-weak-B-flow} by regarding the
generalized KR flow as a perturbation of KR flow.
The initial value may vary, as needed in Section~\ref{sec:main-proof}.
\endgroup

{
Fix the background $(\omega_0,\bar\phi_0)$ of
Section~\ref{sec:preliminaries}. Let each triple $(u_{0,j},0,B_{0,j})$
determine a smooth toric generalized K\"ahler structure, and assume
\begin{equation}\label{eq:uniform-initial-convergence}
 u_{0,j}\longrightarrow u_0\quad\text{in }L^1(P),
 \qquad
 \varphi_0:=u_0^*-\bar\phi_0\in\mathcal E^1_\T,
\end{equation}
and
\begin{equation}\label{eq:Q0-def}
 \sup_j|B_{0,j}|<\infty,\qquad
 \frac12\sup_j\sup_M|b_{0,j}|_{g_{0,j}}^2\leq Q<\infty.
\end{equation}
Let $u_j(s)$ be the corresponding smooth flows, with
$B_j(s)=e^{-2s}B_{0,j}$.  In the fixed-$J$ gauge, denote their
metrics and $b$-fields by $(g_j,b_j)$, and write
\[
 \omega_j(V,W)=g_j(JV,W),\qquad
 \bar\omega_j(s)=\omega_0+\ddcJ\varphi_j(s).
\]
Choose the additive constants of $\varphi_j$ by
\eqref{eq:manifold-scalar-flow}, with
$\varphi_j(0)=u_{0,j}^*-\bar\phi_0$.
Norms of $b_j$ and the Chern torsion are taken with respect to $g_j$.
\par}

\begingroup
Let $\Omega$ be a toric K\"ahler form with normalized moment map
$\mu_\Omega=(\mu_\Omega^1,\ldots,\mu_\Omega^m)$.
For $B=(B_{ab})$, the type-$B$ construction in the $A=0$ model is
\cite[Section~3.2.2, equation~(3.15)]{ASU-toric}
\begin{equation}\label{eq:explicit-B-reconstruction}
 F=\Omega+\sum_{a,b}B_{ab}\,
       d\mu_\Omega^a\wedge d\mu_\Omega^b,\qquad
 \omega_J=F_J^{1,1}.
\end{equation}
For $\tau\in\TC$, the normalized moment map of $\tau^*\Omega$
is $\mu_\Omega\circ\tau$. Thus \eqref{eq:explicit-B-reconstruction}
shows that a time-independent torus pullback preserves $B$.
\endgroup

Write $\phi_j=\bar\phi_0+\varphi_j$ and
$H_j=D_y^2\phi_j$. Recall that $q_j=-\log\det(I+\ii H_jB_j)$ and
\begin{equation}\label{eq:q-volume-ratio}
 \frac{\omega_j^m}{\bar\omega_j^m}=e^{-q_j}.
\end{equation}

\begin{lemma}\label{lem:determinant-error}
 \begin{equation}\label{eq:q-bound}
  \log\left(1+\frac{|b_j|_{g_j}^2}{2}\right)
  \leq q_j\leq\frac{|b_j|_{g_j}^2}{2}\leq Q.
 \end{equation}
\end{lemma}

\begingroup
\begin{proof}
Apostolov--Streets--Ustinovskiy
\cite[proof of Lemma~4.8]{ASU-toric} proved that
\begingroup
\begin{equation}\label{eq:b-norm-trace}
 m+|b_j|_{g_j}^2
 =\tr\bigl((I+B_jH_jB_jH_j)^{-1}\bigr).
\end{equation}
\endgroup
Set $S=H_j^{\frac12}B_jH_j^{\frac12}$, with block parameters
$0\leq s_\alpha<1$ as in \eqref{eq:skew-normal-form}.
Since
\[
 B_jH_jB_jH_j=H_j^{-\frac12}S^2H_j^{\frac12},\qquad
 \det(I+\ii H_jB_j)=\det(I+\ii S),
\]
and each two-dimensional block of $S$ has square $-s_\alpha^2I_2$,
we obtain
\begin{equation}\label{eq:q-b}
 \begin{aligned}
 q_j&=-\sum_\alpha\log(1-s_\alpha^2),\\
 |b_j|_{g_j}^2
 &=\tr\bigl((I+S^2)^{-1}\bigr)-m
 =2\sum_\alpha\frac{s_\alpha^2}{1-s_\alpha^2}.
 \end{aligned}
\end{equation}
These identities extend to $M$ by smoothness.
The inequalities
\[
 \log\left(1+\sum_\alpha a_\alpha\right)
 \leq\sum_\alpha\log(1+a_\alpha)
 \leq\sum_\alpha a_\alpha,\qquad a_\alpha\geq0,
\]
applied with $a_\alpha=\frac{s_\alpha^2}{1-s_\alpha^2}$,
give the first two bounds in \eqref{eq:q-bound}.
The last follows from \eqref{eq:Q0-def} and the
maximum-principle estimate \cite[Proposition~4.10]{ASU-toric}:
\[
 \sup_{M\times[0,\infty)}|b_j|_{g_j}^2
 \leq\sup_M|b_{0,j}|_{g_{0,j}}^2\leq2Q.
\]
\end{proof}
\endgroup

We consider the scalar equation
\eqref{eq:manifold-scalar-flow} with initial K\"ahler potentials
$\varphi_j(0)$ associated to $u_{0,j}$ in
\eqref{eq:uniform-initial-convergence}.
We need to establish the following a priori estimates for its
solutions $\varphi_j$.

\begin{proposition}\label{prop:uniform-smoothing}
Under \eqref{eq:uniform-initial-convergence} and
\eqref{eq:Q0-def}, for every $0<\varepsilon<S<\infty$ and $k\geq0$,
\begin{equation}\label{eq:positive-time}
 \|\varphi_j\|_{C^k([\varepsilon,S]\times M)}
 \leq C_{\varepsilon,S,k}.
\end{equation}
Here $C_{\varepsilon,S,k}$ is independent of $j$.
\end{proposition}

\begingroup
To obtain a $C^0$-estimate, let $\psi_j$ solve the normalized KR
flow \eqref{eq:normalized-KRF} with background
$(\omega_0,h_0)$ and initial value $\psi_j(0)=\varphi_j(0)$,
and write $\omega_{\psi_j}=\omega_0+\ddcJ\psi_j$.
By the maximum principle, we have the following comparison.
\endgroup
 
\begin{lemma}
\label{lem:translated-comparison}
For every $j$ and $s\geq0$,
 \begin{equation}\label{eq:comparison}
  \psi_j(s)\leq\varphi_j(s)
  \leq\psi_j(s)+\frac{Q}{2}(e^{2s}-1).
 \end{equation}
\end{lemma}

\begin{proof}
At a positive spacetime maximum of
$f:=e^{-2s}(\psi_j-\varphi_j)-s\delta$,
where $\delta>0$, we have
$\omega_{\psi_j}\leq\bar\omega_j$.  Hence
\[
 0\leq \frac{\partial f}{\partial s}=e^{-2s}
 \left(\log\frac{\omega_{\psi_j}^m}{\bar\omega_j^m}-q_j\right)
 -\delta\leq-\delta,
\]
which is impossible.  Since both potentials agree at
$s=0$, we obtain $\psi_j\leq\varphi_j$ by letting $\delta\to0^+$.

For the upper bound, set
$a(s)=\frac{Q}{2}(e^{2s}-1)$.  Then $a(0)=0$ and
$a'=2a+Q$, so $\psi_j+a$ satisfies the same scalar equation
as $\psi_j$, with the additional term $Q$ on the right.
Applying the same argument to
$e^{-2s}(\varphi_j-\psi_j-a)-s\delta$, now using
$q_j\leq Q$, proves the upper bound.
\end{proof}

\begingroup
By Lemma~\ref{lem:translated-comparison} together with
Corollary~\ref{cor:DNL-normalized}, we get the following
$C^0$-estimate for the solutions of \eqref{eq:manifold-scalar-flow}.
Indeed, by \eqref{eq:Legendre-L1-convergence},
$\psi_j(0)=\varphi_j(0)\to\varphi_0$ in $L^1(M,\omega_0^m)$.
Since $\varphi_0\in\mathcal E^1_\T$, we deduce from
Corollary~\ref{cor:DNL-normalized} and \eqref{eq:comparison} that
\begin{equation}\label{eq:C0-phi}
 \|\varphi_j\|_{C^0([\frac{\varepsilon}{4},S]\times M)}
 \leq C_{\varepsilon,S},
 \qquad 0<\varepsilon<S<\infty.
\end{equation}
The constant is independent of $j$.

As an application of \eqref{eq:C0-phi}, we also have the following
$C^0$-estimate for $\dot\varphi_j$.
\endgroup

\begin{lemma}\label{lem:time-derivative-bound}
 For every $0<\varepsilon<S<\infty$, there is a constant
 $C_{\varepsilon,S}$, independent of $j$, such that
 \begin{equation}\label{eq:dot-bound}
  |\dot\varphi_j|\leq C_{\varepsilon,S}
  \quad\text{on }[\frac{\varepsilon}{2},S]\times M.
 \end{equation}
\end{lemma}

\begin{proof}
 Let $\mathcal L_j=\partial_s-\Delta^C_{g_j}$.
\begingroup
The evolution equations for $\dot\varphi_j$ and $\varphi_j$ are
as follows \cite[Lemma~4.8 and equation~(4.17)]{ASU-toric}:
\endgroup
\begin{equation}\label{eq:dotphi-evol}
  \mathcal L_j\dot\varphi_j
  =2\dot\varphi_j-2|b_j|^2,\qquad
 \end{equation}
 and
 \begin{equation}\label{eq:phi-evol}
  \mathcal L_j\varphi_j
  =\dot\varphi_j-\tr_{\omega_j}\bar\omega_j
   +\tr_{\omega_j}\omega_0.
 \end{equation}
\begingroup
On the dense orbit, \eqref{eq:toric-Hermitian-form} gives the
matrix of $g_j$ in the ordered real frame
$(\partial_{y_1},\ldots,\partial_{y_m},
  \partial_{\theta_1},\ldots,\partial_{\theta_m})$:
\[
 g_j\sim
 \begin{pmatrix}
  H_j&-H_jB_jH_j\\
  H_jB_jH_j&H_j
 \end{pmatrix}.
\]
By the Schur complement formula, both diagonal blocks of
$g_j^{-1}$ are $(H_j+H_jB_jH_jB_jH_j)^{-1}$.
Since $\bar\omega_j=\sum_{a,b}(H_j)_{ab}\,dy_a\wedge d\theta_b$
and the Hermitian trace is half the corresponding real trace,
we obtain
\[
 \begin{aligned}
  \tr_{\omega_j}\bar\omega_j
  &=\frac12\left\{
    2\tr\bigl((H_j+H_jB_jH_jB_jH_j)^{-1}H_j\bigr)
    \right\}\\
  &=\tr\bigl((I+B_jH_jB_jH_j)^{-1}\bigr).
 \end{aligned}
\]
By \eqref{eq:b-norm-trace} and \eqref{eq:q-bound}, we have
 \begin{equation}\label{eq:bartrace}
  \tr_{\omega_j}\bar\omega_j
  =m+|b_j|_{g_j}^2\leq m+2Q=:C.
 \end{equation}
The identity extends to all of $M$ by smoothness, and $C$ is
independent of $j$ and $s$.
\endgroup

 Set $s_0=\frac{\varepsilon}{4}$ and $\tau=s-s_0$.  For
 \[
 \mathcal R_j
  =\tau\dot\varphi_j-(1+2\tau)\varphi_j-C\tau
 \]
 it follows from equations~\eqref{eq:dotphi-evol}
 and \eqref{eq:phi-evol} that
 \[
 \begin{split}
  \mathcal L_j\mathcal R_j
  ={}&-2\tau|b_j|^2-2\varphi_j
     +(1+2\tau)\tr_{\omega_j}\bar\omega_j\\
    &-(1+2\tau)\tr_{\omega_j}\omega_0-C.
 \end{split}
 \]
 Equations~\eqref{eq:C0-phi} and \eqref{eq:bartrace} bound the
 positive terms on the right for $0\leq\tau\leq S$.  Choose $C$
 so that $\mathcal L_j\mathcal R_j\leq0$.
 Applying the maximum principle on
 $M\times[s_0,S]$, we obtain
 \[
  \sup_{M\times[s_0,S]}\mathcal R_j
  \leq\sup_M\mathcal R_j(s_0)
  =\sup_M(-\varphi_j(s_0))\leq C.
 \]
 Since $\tau\geq\frac{\varepsilon}{4}$ on
 $[\frac{\varepsilon}{2},S]$, we obtain from \eqref{eq:C0-phi} that
 \begin{equation}\label{eq:dot-upper}
  \dot\varphi_j\leq C_{\varepsilon,S}
  \quad\text{on }[\frac{\varepsilon}{2},S].
 \end{equation}
 
\begingroup
For the lower bound, set
$\mathcal P_j=\tau\dot\varphi_j+\varphi_j$.
By the volume identity,
\begin{equation}\label{eq:Hermitian-volume}
 \log\frac{\omega_j^m}{\omega_0^m}
 =\dot\varphi_j-2\varphi_j+2h_0-2q_j.
\end{equation}
The remaining terms are bounded by \eqref{eq:C0-phi} and
\eqref{eq:q-bound}, so the arithmetic--geometric mean inequality implies
$\tr_{\omega_j}\omega_0\geq c e^{-\frac{\dot\varphi_j}{m}}$.
Also,
\[
 \mathcal L_j\mathcal P_j
 =(2+2\tau)\dot\varphi_j-2\tau|b_j|^2
   -\tr_{\omega_j}\bar\omega_j+\tr_{\omega_j}\omega_0.
\]
At a spacetime minimum $(x_0,s_0+\tau_0)$ with $\tau_0>0$, put
$\zeta=\dot\varphi_j(x_0,s_0+\tau_0)$.
Using \eqref{eq:bartrace} and the bound on $b_j$, we find
\[
 0\geq\mathcal L_j\mathcal P_j
 \geq c e^{-\frac{\zeta}{m}}+(2+2\tau_0)\zeta-C.
\]
The right-hand side tends to $+\infty$ as $\zeta\to-\infty$,
uniformly for $0\leq\tau_0\leq S$.
Thus $\zeta$ is bounded below, and so is $\mathcal P_j$ at
this minimum.  At $\tau=0$, the same bound follows from
$\mathcal P_j=\varphi_j$ and \eqref{eq:C0-phi}.
Hence $\tau\dot\varphi_j=\mathcal P_j-\varphi_j\geq-C$.
Since $\tau\geq\frac{\varepsilon}{4}$ on
$[\frac{\varepsilon}{2},S]$, this proves the lower bound.
Together with \eqref{eq:dot-upper}, this proves \eqref{eq:dot-bound}.
\endgroup
\end{proof}

To obtain the $C^2$-estimate, we prove the following
metric comparison.

\begin{lemma}\label{lem:uniform-metric-equivalence}
 For every $0<\varepsilon<S<\infty$, there is a constant
 $C_{\varepsilon,S}\geq1$, independent of $j$, such that on
 $[\varepsilon,S]\times M$,
 \[
  C_{\varepsilon,S}^{-1}\omega_0
  \leq\omega_j\leq C_{\varepsilon,S}\omega_0,
 \]
 and
 \begin{equation}\label{eq:uniform-equivalence}
  C_{\varepsilon,S}^{-1}\omega_0
  \leq\bar\omega_j\leq C_{\varepsilon,S}\omega_0.
 \end{equation}
\end{lemma}

\begin{proof}
 Let $T_j^c$ denote the torsion tensor of the Chern connection of
 the Hermitian metric $(g_j,J)$.
We use the following evolution inequalities
\cite[Equations~(4.15)--(4.16)]{ASU-toric}:
 \begin{equation}\label{eq:trace-evol}
  \mathcal L_j\log\tr_{\omega_0}\omega_j
  \leq|T_j^c|^2+C\tr_{\omega_j}\omega_0+C,
  \qquad
  \mathcal L_j|b_j|^2\leq-|T_j^c|^2.
 \end{equation}
\begingroup
Set $\tau=s-\frac{\varepsilon}{2}$ and consider 
 \[
 \mathcal Q_j
 =\tau\log\tr_{\omega_0}\omega_j
 +2\tau|b_j|^2-K\varphi_j-C\tau,
 \]
 where $K,C>0$ will be chosen below.
From equations~\eqref{eq:phi-evol} and \eqref{eq:trace-evol}, we compute
 \[
 \begin{split}
  \mathcal L_j\mathcal Q_j
  \leq{}&
  \log{\tr_{\omega_0}\omega_j}
  -\tau|T_j^c|^2
  -\bigl(K-C_0\tau\bigr)
    \tr_{\omega_j}\omega_0\\
  &+2|b_j|^2-K\dot\varphi_j
    +K\tr_{\omega_j}\bar\omega_j+C_0\tau-C,
 \end{split}
 \]
where $C_0$ bounds the fixed constants in \eqref{eq:trace-evol}.
Choose $K>2C_0S+1$ and enlarge $C$.
Using \eqref{eq:bartrace}, \eqref{eq:dot-bound}, and the bound
on $b_j$, we obtain
 \begin{equation}\label{eq:trace-test}
  \mathcal L_j\mathcal Q_j
  \leq\log{\tr_{\omega_0}\omega_j}-c\,\tr_{\omega_j}\omega_0+C.
 \end{equation}

Combining equations~\eqref{eq:Hermitian-volume}, \eqref{eq:C0-phi},
\eqref{eq:dot-bound}, and \eqref{eq:q-bound}, we obtain
$C^{-1}\leq\frac{\omega_j^m}{\omega_0^m}\leq C$.
The elementary trace--determinant inequality therefore yields
\[
 \tr_{\omega_0}\omega_j
 \leq C(1+\tr_{\omega_j}\omega_0)^{m-1},\qquad
 \log\tr_{\omega_0}\omega_j
 \leq \eta\,\tr_{\omega_j}\omega_0+C_\eta
\]
for every $\eta>0$.
At a positive-time maximum of $\mathcal Q_j$, choose $\eta<c$
in \eqref{eq:trace-test}. Then
$0\leq C-(c-\eta)\tr_{\omega_j}\omega_0$, which bounds both
traces at the maximum point. The other terms in $\mathcal Q_j$ are bounded
by \eqref{eq:C0-phi} and \eqref{eq:q-bound}.
At $\tau=0$, we have $\mathcal Q_j=-K\varphi_j$, so its maximum
on the whole cylinder is uniformly bounded.
Since $\tau\geq\frac{\varepsilon}{2}$ on $[\varepsilon,S]$,
we obtain $\tr_{\omega_0}\omega_j\leq C$ on
$[\varepsilon,S]\times M$.
The reverse metric inequality follows from the determinant lower bound.

 By \eqref{eq:bartrace},
 $\bar\omega_j\leq C\omega_j$, while
 \[
  \log\frac{\bar\omega_j^m}{\omega_0^m}
  =\dot\varphi_j-2\varphi_j+2h_0-q_j
 \]
 is bounded above and below.  The upper metric bound and the
 determinant lower bound for $\bar\omega_j$ prove
 \eqref{eq:uniform-equivalence}.
\endgroup
\end{proof}

\begingroup
By the $C^2$-estimate in
Lemma~\ref{lem:uniform-metric-equivalence}, we can derive
higher-order estimates and then finish the proof of
Proposition~\ref{prop:uniform-smoothing}.
\endgroup

\begingroup
\begin{proof}[Proof of Proposition~\ref{prop:uniform-smoothing}]
By Lemma~\ref{lem:uniform-metric-equivalence} and \eqref{eq:q-bound},
$g_j$ is uniformly equivalent to the reference metric and
$\beta_j=-\ii\pi_J^{2,0}b_j$ is uniformly bounded on
$[\frac{\varepsilon}{2},S]\times M$.
Using formula~\eqref{eq:JS-generalized-metric-preliminary}, we obtain
\[
 C^{-1}\widetilde{\mathcal G}\leq\mathcal G_j
 \leq C\widetilde{\mathcal G},
\]
where $\widetilde{\mathcal G}$ is the direct-sum metric induced by
$\omega_0$.
Rescale from $s=\frac{\varepsilon}{2}$ as in
Lemma~\ref{lem:JS-convention}; the same equivalence holds on this
finite interval. Applying Theorem~\ref{thm:JS-local} on a finite
coordinate cover of $M$, we obtain uniform bounds for all spatial
derivatives of $g_j$ and $\beta_j$ on $[\varepsilon,S]\times M$.
The mixed space--time bounds follow from the flow equations.

By \eqref{eq:phi-evol} and \eqref{eq:bartrace},
\[
 \Delta^C_{g_j}\varphi_j
 =m+|b_j|_{g_j}^2-\tr_{\omega_j}\omega_0=:f_j.
\]
Since $b_j=\ii(\beta_j-\overline{\beta_j})$, these estimates also
bound $f_j$ uniformly in every $C^{k,\alpha}(M)$-norm.
The uniform ellipticity of $\Delta^C_{g_j}$ and Schauder estimates,
together with \eqref{eq:C0-phi}, therefore yield
\[
 \|\varphi_j(s)\|_{C^{k+2,\alpha}(M)}
 \leq C\bigl(\|f_j(s)\|_{C^{k,\alpha}(M)}
             +\|\varphi_j(s)\|_{C^0(M)}\bigr)
 \leq C_{\varepsilon,S,k,\alpha}
\]
for every $k\geq0$, $0<\alpha<1$, and $s\in[\varepsilon,S]$,
with constants independent of $j$.
Finally, the time-derivative bounds follow from the scalar equation
\eqref{eq:manifold-scalar-flow} and the volume identity
\eqref{eq:q-volume-ratio}, completing the proof of \eqref{eq:positive-time}.
\end{proof}
\endgroup

\begingroup
By Proposition~\ref{prop:uniform-smoothing}, to finish the proof of
Theorem~\ref{thm:conditional-weak-B-flow} it remains to show that
the limiting solution attains the prescribed initial data and to
prove the uniqueness.

\begin{proof}[Completion of the proof of Theorem~\ref{thm:conditional-weak-B-flow}]
Applying Proposition~\ref{prop:uniform-smoothing} with $B_{0,j}=B_0$,
we obtain smooth precompactness on compact subsets of
$(0,\infty)\times M$.
The bound $q_j\leq Q$ and \eqref{eq:q-b} imply
$s_\alpha^2\leq1-e^{-Q}$ for each skew-symmetric block.  Hence the smallest
eigenvalue of $I+\ii H_j^{\frac12}B_tH_j^{\frac12}$ is at least
$1-\sqrt{1-e^{-Q}}>0$.  Together with the metric bounds, this
shows that every smooth limit remains admissible.  The same bound
holds on each face by continuity up to the toric boundary.
Thus any subsequential limit $\varphi(t)$ solves
\eqref{eq:manifold-scalar-flow} with $B_t=e^{-2t}B_0$.

Next we use the comparison in Lemma~\ref{lem:translated-comparison}
to prove the convergence to the prescribed initial data.
Let $\chi(t)$ be the normalized weak KR flow with
initial potential $\varphi_0$.  Passing to the limit in the
comparison for the smooth approximations gives
\begin{equation}\label{eq:weak-B-KRF-sandwich}
 \chi(t)\leq\varphi(t)
 \leq\chi(t)+\frac Q2(e^{2t}-1),
\end{equation}
by Lemma~\ref{lem:translated-comparison} and
Corollary~\ref{cor:DNL-normalized}.
Since $\varphi_0\in\mathcal E^1_\T$, the potentials $\chi(t)$
converge to $\varphi_0$ in $L^1(M,\omega_0^m)$ and their
Aubin--Mabuchi energies converge as $t\to0^+$
\cite[Proposition~5.2]{GuedjZeriahiKRF}.
Put $v(t)=(\bar\phi_0+\chi(t))^*$.
By convexity and Legendre duality, we have local uniform convergence
$v(t)\to u_\infty$ on $P^\circ$.
Supporting affine functions at an interior point provide a common
lower bound on $P$, while convergence of the integrals follows from
\eqref{eq:toric-AM-energy}.  Scheff\'e's lemma therefore yields
$v(t)\to u_\infty$ in $L^1(P)$.
Taking Legendre transforms in \eqref{eq:weak-B-KRF-sandwich},
we obtain
\[
 v(t)-\frac Q2(e^{2t}-1)\leq u(t)\leq v(t).
\]
Consequently, as $t\to0^+$,
$u(t)\to u_\infty$ in $L^1(P)$, and
\eqref{eq:weak-B-KRF-sandwich} also implies
$\varphi(t)\to\varphi_0$ in $L^1(M,\omega_0^m)$.

Now we prove the uniqueness within the class of approximation limits.
Let $\varphi$ and $\psi$ be two approximation limits with the same
initial data $(\varphi_0,B_0)$, and choose $Q$ common to the two
approximating sequences.  By \eqref{eq:weak-B-KRF-sandwich}, both
limits lie between the same $\chi(t)$ and
$\chi(t)+\frac Q2(e^{2t}-1)$.  Hence, for every $\delta>0$,
\[
 \|\varphi(\delta)-\psi(\delta)\|_{C^0(M)}
 \leq\frac Q2(e^{2\delta}-1).
\]
For a fixed $B$, the operator
$H\mapsto-\log\det(H^{-1}+\ii B)$ is increasing on ordered
admissible Hessians, and its linearization is the Chern Laplacian
of the reconstructed metric~\cite[Lemma~4.7]{ASU-toric}.
Since both solutions have the same parameter $B_t=e^{-2t}B_0$,
the maximum principle on $[\delta,t]\times M$ gives
\[
 \begin{aligned}
 \|\varphi(t)-\psi(t)\|_{C^0(M)}
 &\leq e^{2(t-\delta)}
       \|\varphi(\delta)-\psi(\delta)\|_{C^0(M)}\\
 &\leq\frac Q2e^{2(t-\delta)}(e^{2\delta}-1).
 \end{aligned}
\]
Letting $\delta\to0^+$ proves $\varphi=\psi$.
The full approximating sequence therefore converges by
the precompactness and the uniqueness.
\end{proof}

The following corollary will be used in the proof of
Theorem~\ref{thm:main} in the next section.

\begin{corollary}\label{cor:vanishing-B-regularization}
Assume \eqref{eq:uniform-initial-convergence} and
\eqref{eq:Q0-def} hold, and suppose that $B_{0,j}\to0$.
Then $\varphi_j$ converges in
$C^\infty_{\mathrm{loc}}((0,\infty)\times M)$ to the normalized
weak KR flow with initial potential $\varphi_0$.
\end{corollary}

\begin{proof}
By Proposition~\ref{prop:uniform-smoothing}, we obtain smooth
positive-time precompactness and a K\"ahler limit
$\omega_0+\ddcJ\varphi_*$.
The Hessian entries
$(H_j)_{ab}=\bar\omega_j(K_a,JK_b)$ have uniform derivative
bounds on each compact positive-time cylinder.
Since $B_j(s)=e^{-2s}B_{0,j}\to0$, it follows that
$q_j=-\log\det(I+\ii H_jB_j)\to0$ smoothly there.
Thus \eqref{eq:manifold-scalar-flow} passes to the normalized
K\"ahler--Ricci flow equation.
Let $\chi$ be its weak solution with initial potential
$\varphi_0$.  By Lemma~\ref{lem:translated-comparison} and
Corollary~\ref{cor:DNL-normalized},
\[
 0\leq\varphi_*(s)-\chi(s)\leq\frac Q2(e^{2s}-1).
\]
Hence $\varphi_*(s)\to\varphi_0$ in $L^1(M,\omega_0^m)$ as
$s\to0^+$.  By the uniqueness in
Corollary~\ref{cor:DNL-normalized}, $\varphi_*=\chi$.
Precompactness then implies convergence of the whole sequence.
\end{proof}
\endgroup

\section{Proof of Theorem \ref{thm:main}}
\label{sec:main-proof}
\label{sec:regularization-rigidity}

\begingroup
In this section we prove Theorem~\ref{thm:main} by showing the
convergence of the triples
$(\tilde{\omega}_t,\widetilde I_t,\widetilde b_t)$ modulo $\TC$.

Recall that Proposition~\ref{prop:toric-compactness} gives
compactness of the normalized symplectic potentials in $L^1(P)$.
Applying Corollary~\ref{cor:vanishing-B-regularization} and using
the dissipation formula \eqref{eq:exact-dissipation} together with
Corollary~\ref{cor:finite-D}, we get the following result.
\endgroup

\begingroup
\begin{proposition}\label{prop:tail-KRS}
For every sequence $t_j\to\infty$, there are a subsequence and
$\tau_j\in\TC$ such that
\[
 \tau_j^*\bar\omega_{t_j+s}\longrightarrow\bar\omega_\infty(s)
 \quad\text{in }C^\infty_{\mathrm{loc}}((0,\infty)\times M).
\]
Every time slice $\bar\omega_\infty(s)$ is a KR soliton.
\end{proposition}

\begin{proof}
{
Fix any sequence $t_j\to\infty$. By
Proposition~\ref{prop:toric-compactness}, we can pass to a
subsequence, still denoted by $(t_j)$, and choose $\tau_j\in\TC$
for which the convergences in \eqref{eq:toric-compactness} hold.}
Keeping $\tau_j$ fixed in $s$, write
\[
 \bar\omega_j(s):=\tau_j^*\bar\omega_{t_j+s}
 =\omega_0+\ddcJ\varphi_j(s),\qquad
 u_j(s)=(\bar\phi_0+\varphi_j(s))^*.
\]
Choose the constants by \eqref{eq:manifold-scalar-flow}, with
$\varphi_j(0)=\tau_j[\varphi_{t_j}]$.
Then $u_j(0)\to u_\infty$ in $L^1(P)$ and
$\varphi_\infty=u_\infty^*-\bar\phi_0\in\mathcal E^1_\T$.
The pullback leaves the type-$B$ matrix unchanged, as noted
after \eqref{eq:explicit-B-reconstruction}.  Thus
\[
 B_j(s)=e^{-2s}B_{0,j},\qquad
 B_{0,j}=e^{-2t_j}B_0\longrightarrow0.
\]
{
Let $(g_j,b_j)$ be the metric and two-form associated to
$(u_j,0,B_j)$ by the toric reconstruction formulas in
Section~\ref{sec:preliminaries}, expressed in the fixed-$J$ gauge.}
By Lemma~\ref{lem:determinant-error},
$\frac12\sup_M|b_j(0)|_{g_j(0)}^2\leq Q$ for a constant
independent of $j$.
Corollary~\ref{cor:vanishing-B-regularization} now yields
\[
 \varphi_j\longrightarrow\varphi_*
 \quad\text{in }C^\infty_{\mathrm{loc}}((0,\infty)\times M),
\]
where $\varphi_*$ is the normalized weak K\"ahler--Ricci flow
from $\varphi_\infty$.
Thus $\bar\omega_\infty(s)=\omega_0+\ddcJ\varphi_*(s)$.

Put $u_*=(\bar\phi_0+\varphi_*)^*$.
Fix $0<a<b<\infty$ and
$P_\delta=\{x\in P:L_r(x)\geq\delta\text{ for every }r\}
\Subset P^\circ$, where $\delta>0$ is small.
\begingroup
Write $\phi_j=\bar\phi_0+\varphi_j$ and
$\phi_*=\bar\phi_0+\varphi_*$. The preceding convergence implies
that $\phi_j\to\phi_*$ locally smoothly in $(s,y)$.
Since the reductions are K\"ahler, their gradient maps are
diffeomorphisms from $\mathbb R^m$ onto $P^\circ$.
Define the inverse maps $y_j(s,x)$ and $y_*(s,x)$ by
\[
 \nabla_y\phi_j(s,y_j(s,x))=x,\qquad
 \nabla_y\phi_*(s,y_*(s,x))=x.
\]
The image of $[a,b]\times P_\delta$ under
$(s,x)\mapsto(s,y_*(s,x))$ is compact. Thus
$D_y^2\phi_*\geq cI$ on a fixed neighborhood of this image for
some $c>0$, and smooth convergence gives
$D_y^2\phi_j\geq\frac c2I$ there for all sufficiently large $j$.
The inverse function theorem with parameter $s$, applied on a
finite cover of this image, therefore gives $y_j\to y_*$ in
$C^\infty([a,b]\times P_\delta)$; strict convexity identifies the
local inverse branches with the global maps $y_j$.
By Legendre duality,
\[
 u_j(s,x)=\langle x,y_j(s,x)\rangle-\phi_j(s,y_j(s,x)).
\]
The smooth convergence of $\phi_j$ and $y_j$ on the relevant
compact sets therefore yields
\[
 u_j\longrightarrow u_*
 \quad\text{in }C^\infty([a,b]\times P_\delta).
\]
\endgroup
Set $U_j=D_x^2u_j$ and
$p_j=\partial_su_j-\langle\xi,x\rangle$, and define $U_*,p_*$
in the same way.
By \eqref{eq:torus-action-Legendre},
$u_j(s)-u_{t_j+s}$ is affine in $x$ with gradient independent
of $s$.  Hence $U_j=D^2u_{t_j+s}$ and
$\nabla p_j=\nabla(\dot u_{t_j+s}-\langle\xi,x\rangle)$.
The dissipation formula~\eqref{eq:exact-dissipation} and
Corollary~\ref{cor:finite-D} imply
\[
 0\leq\int_a^b\!\int_{P_\delta}
 \rho_\xi U_j^{-1}(\nabla p_j,\nabla p_j)\,\dd x\,\dd s
 \leq\int_{t_j+a}^\infty\cD(t)\,\dd t\longrightarrow0.
\]
The integrands converge smoothly.  Since $\rho_\xi>0$ and
$U_*>0$, we obtain $\nabla p_*=0$ on this cylinder and,
letting $\delta\to0^+$, on $[a,b]\times P^\circ$.
Differentiating \eqref{eq:intro-flow} in $x$, we have
\[
 \nabla p_j
 =\nabla\log\det(U_j+\ii B_j)-2U_jx-\xi.
\]
Passing to the limit and integrating in $x$, we find
\[
 \log\det D_x^2u_*+2u_*
 -2x\cdot\nabla u_*-\langle\xi,x\rangle=c(s).
\]
Subtracting $\frac{c(s)}2$ from $u_*(s)$, we obtain the soliton equation
without changing the K\"ahler form. Thus every positive-time slice
is a KR soliton.
\end{proof}
\endgroup

\begingroup
The following is due to \cite[Theorem~1.3 and Section~4]{HuangZhu}.

\begin{lemma}\label{lem:toric-local-stability}Let  $\omega_{\KRS}$ be the KR soliton associated to   $\xi$ in  Definition \ref{def:M}.
Fix $0<\alpha<1$.  There is  an $\epsilon_0>0$ with the following
property:  let $\bar\omega_s$ be  the K\"ahler form associated to the K\"ahler reduction equation (\ref{eq:manifold-scalar-flow}) with $B_s=e^{-2s}B_*$ and
$\bar\omega_0=\omega_{\KRS}+\ddcJ\psi_0$.
If, after subtracting a constant,
\[
 \|\psi_0\|_{C^{2,\alpha}}+|B_*|<\epsilon_0,
\]
then there is a smooth family $\sigma_s\in\TC$ such that
$\sigma_s^*\bar\omega_s\to\omega_{\KRS}$ in $C^\infty$ as
$s\to\infty$.
\end{lemma}

\begin{proof}
By the stability theorem \cite[Theorem~1.3]{HuangZhu},
it suffices to show that the normalizing automorphisms can be
chosen smoothly in $\TC$. Their stability argument yields a
normalized family $\widehat\omega_s$, precompact in
$C^\infty(M)$, whose limits as $s\to\infty$ are KR solitons.
For torus-invariant potentials, the kernel of the linearized
soliton equation modulo constants is spanned by the moment-map
components. Thus the normalization in
\cite[Lemma~4.7 and Proposition~4.8]{HuangZhu} can be performed
by logarithmic translations $T_a:y\mapsto y+a$.
They preserve $B_s=e^{-2s}B_*$ and
commute with the soliton vector field. We may therefore choose
this normalization so that
$\widehat\omega_s=T_{b_s}^*\bar\omega_s$ with
$b_s\in\mathfrak t$. Let $u_s$ denote the Legendre potential of
$\bar\omega_s$, and write
$v_s=u_s-\langle b_s,x\rangle$ for that of
$\widehat\omega_s$, up to a constant.

To normalize the gradient at the origin, set
\[
a_s=\nabla u_s(0)-\nabla u_{\KRS}(0),\qquad \sigma_s=T_{a_s}.
\]
Then $\sigma_s$ depends smoothly on $s$, and the translated
Legendre potential $u_s-\langle a_s,x\rangle$ has gradient
$\nabla u_{\KRS}(0)$ at the origin.
Since $\nabla v_s$ is the inverse of the normalized moment map
in logarithmic coordinates, the inverse function theorem at the
K\"ahler limits gives $\nabla v_s(0)=O(1)$, as $0\in P^\circ$.
Hence
\[
\begin{aligned}
	a_s-b_s
	&=\nabla v_s(0)-\nabla u_{\KRS}(0)=O(1),\\
	\sigma_s^*\bar\omega_s
	&=T_{a_s-b_s}^*\widehat\omega_s.
\end{aligned}
\]
The bounded additional translations preserve precompactness.
By Proposition~\ref{prop:M-variational-structure}, the gradient
normalization makes every limiting soliton potential differ
from $u_{\KRS}$ only by a constant, so every limiting form is
$\omega_{\KRS}$. Precompactness therefore yields
$\sigma_s^*\bar\omega_s\to\omega_{\KRS}$ in $C^\infty(M)$
as $s\to\infty$.
\end{proof}

By Proposition~\ref{prop:tail-KRS} and the stability result above,
we obtain the following convergence of $\bar\omega_t$.

\begin{proposition}\label{prop:reduction-convergence}
There is a smooth family $\gamma_t\in\TC$ such that
\[
 \gamma_t^*\bar\omega_t\longrightarrow\omega_{\KRS}
 \quad\text{in }C^\infty.
\]
\end{proposition}

\begin{proof}
Apply Proposition~\ref{prop:tail-KRS} to one sequence
$t_j\to\infty$ and evaluate at $s=1$.
Then $\tau_j^*\bar\omega_{t_j+1}\to\omega_*$ smoothly, where
$\omega_*$ is a KR soliton.
By Proposition~\ref{prop:M-variational-structure} and
\eqref{eq:torus-action-Legendre}, there is a fixed $\eta\in\TC$
such that $\eta^*\omega_*=\omega_{\KRS}$.
The mean-zero potentials of
$\eta^*\tau_j^*\bar\omega_{t_j+1}$ relative to $\omega_{\KRS}$
tend to zero in $C^{2,\alpha}$ by elliptic estimates.
Also $B_{t_j+1}\to0$.
Fix one sufficiently large $j$, and put $T=t_j+1$ and
$\tau=\tau_j\circ\eta$.
Since $\tau$ is independent of time, the family
$\tau^*\bar\omega_{T+s}$ satisfies the same reduced flow with
initial parameter $B_T$ and meets the smallness condition in
Lemma~\ref{lem:toric-local-stability}.
It follows that
$\sigma_s^*\tau^*\bar\omega_{T+s}\to\omega_{\KRS}$ smoothly.
Set $\gamma_t=\tau\circ\sigma_{t-T}$ for all sufficiently large
$t$.  Its logarithm is a smooth curve in $\mathfrak t$, which
can be extended smoothly over the remaining finite time interval.
\end{proof}
\endgroup

\begingroup
We now recover the full bi-Hermitian data from the K\"ahler
reductions. In fact, we can prove the following result.

\begin{proposition}\label{prop:reconstruction-convergence}
Suppose that $\gamma_t\in\TC$ is a smooth family and
\[
 \gamma_t^*\bar\omega_t\longrightarrow\bar\omega_\infty
 \quad\text{in }C^\infty(M),
\]
where $\bar\omega_\infty$ is a K\"ahler form.
Then, with $g_\infty=\bar\omega_\infty(\cdot,J\cdot)$,
\[
 (\gamma_t^*\widetilde g_t,\gamma_t^*\widetilde I_t,
       \gamma_t^*\widetilde b_t)
 \longrightarrow(g_\infty,J,0)
 \quad\text{in }C^\infty(M).
\]
If the reductions converge exponentially with rate $\lambda_0>0$
in every $C^k$-norm, the full data converge with rate
$\min\{\lambda_0,2\}$.
\end{proposition}

\begin{proof}
Recall that the reconstruction formulas in Section~\ref{sec:preliminaries}
commute with pullback by $\TC$, which leaves $A_t,B_t$ unchanged.
We first prove convergence of the associated symplectic forms.
At each time, write $F$ for the symplectic form of the pulled-back
generalized K\"ahler structure and $\mu$ for the normalized moment
map of $\gamma_t^*\bar\omega_t$.
The assumed smooth convergence of the reductions implies smooth
convergence of these moment maps; in particular, all their
derivatives are uniformly bounded.

If $A_t=0$, formula~\eqref{eq:explicit-B-reconstruction} reads
\[
 F=\gamma_t^*\bar\omega_t
       +\sum_{a,b}(B_t)_{ab}\,d\mu^a\wedge d\mu^b.
\]
Thus the difference between $F$ and $\gamma_t^*\bar\omega_t$
is $O(|B_t|)$ in every $C^k$-norm.
For general $A_t$, the type-$A$ construction changes the real
holomorphic coordinates by
\[
 y\longmapsto y+A_t(\mu(y)),
\]
as in \cite[equation~(3.13) and Lemma~3.6]{ASU-toric}.
This coordinate change extends smoothly over $M$: it is induced
by the real part of the global $\TC$-action with the
point-dependent parameter $A_t(\mu)$.
Since the derivatives of $\mu$ are uniformly bounded and
$A_t\to0$, this diffeomorphism and its inverse differ from the
identity by $O(|A_t|)$ in every $C^k$-norm.
In the fixed-$J$ gauge, $F$ is obtained by pulling back the
type-$B$ form displayed above by this inverse coordinate change.
Consequently, for all sufficiently large $t$,
\[
 \bigl\|F-\gamma_t^*\bar\omega_t\bigr\|_{C^k(M)}
 \leq C_k\bigl(|A_t|+|B_t|\bigr)=O(e^{-2t}).
\]
It follows that $F\to\bar\omega_\infty$ in $C^\infty(M)$.

The remaining tensors are recovered from $F$ and the fixed
complex structure $J$ by the algebraic identities
\cite[equation~(1.3)]{ASU-toric}
\[
 \gamma_t^*\widetilde g_t=-(FJ)_{\mathrm{sym}},\qquad
 \gamma_t^*\widetilde b_t=-(FJ)_{\mathrm{skew}},\qquad
 \gamma_t^*\widetilde I_t=-F^{-1}J^*F.
\]
Because $\bar\omega_\infty$ is nondegenerate, inversion is smooth
near the limit. Since it is also $J$-K\"ahler, the right-hand
sides converge respectively to $g_\infty$, $0$, and $J$.
This proves the asserted smooth convergence.
If the reductions converge at rate $e^{-\lambda_0t}$, the estimate
for $F$ and the same identities bound the total error in every
$C^k$-norm by $O(e^{-\lambda_0t})+O(e^{-2t})$, proving the rate
$\min\{\lambda_0,2\}$.
\end{proof}

\begin{proof}[Proof of Theorem~\ref{thm:main}]
By Proposition~\ref{prop:reduction-convergence}, the K\"ahler
reductions converge to $\omega_{\KRS}$ after pullback by $\gamma_t$.
Then Proposition~\ref{prop:reconstruction-convergence} gives
convergence of the full data, proving
\eqref{eq:main-KRS-convergence}, with
$g_{\KRS}=\omega_{\KRS}(\cdot,J\cdot)$.
Here $\gamma_t^*\tilde{\omega}_t
=(\gamma_t^*\widetilde g_t)(J\cdot,\cdot)$, since $\gamma_t^*J=J$.

In case that  the Futaki invariant vanishes,  $\xi=0$ and
$\omega_{\KRS}$ is  a K\"ahler--Einstein metric.
Choose $T$ and $\tau$ as in the proof of
Proposition~\ref{prop:reduction-convergence}, with $T$ sufficiently
large.  By the result  of \cite[Theorem~1.2]{HuangZhu}, there are a fixed
$\sigma\in\Aut_0(M,J)$ and $\lambda_0>0$ such that

\[
 \|\sigma^*\tau^*\bar\omega_t-\omega_{\KRS}\|_{C^k}
 \leq C_k e^{-\lambda_0(t-T)},\qquad t\geq T.
\]
Thus $\bar\omega_t$ converges exponentially to
$\omega_{\KE,\infty}
=(\tau^{-1})^*(\sigma^{-1})^*\omega_{\KRS}$,
which is $\T$-invariant as well as $\bar\omega_t$.
Applying Proposition~\ref{prop:reconstruction-convergence} with
$\gamma_t=\operatorname{id}$, we obtain  the exponential convergences of metrics, complex structures and $b$-fields with rate
$\min\{\lambda_0,2\}$.
\end{proof}
\endgroup

\end{document}